\documentclass[12pt]{amsart}

\usepackage{yfonts, amsfonts}
\usepackage[a4paper]{geometry}
\usepackage[centertags]{amsmath}
\usepackage{amsmath, amsopn, amsthm, amssymb}
\usepackage{graphicx, subfigure}
\usepackage{times}
\usepackage{color}
\usepackage[final]{changes}
\definechangesauthor[name={Wael}, color=purple]{WB}
\definechangesauthor[name={Chris}, color=blue]{CB}
\setremarkmarkup{(#2)}
\usepackage[mathcal]{euscript}
\usepackage{mathtools}
\usepackage{bbold}
\usepackage{hyperref}
\usepackage{mathrsfs}

\numberwithin{equation}{section}

{\bf}{\it}
\newtheorem{proposition}{Proposition}[section]

\theoremstyle{definition}
\newtheorem{theorem}{Theorem}[section]%

\newtheorem{corollary}[proposition]{Corollary}
\newtheorem{lemma}[proposition]{Lemma}

\newtheorem{definition}{Definition}

\theoremstyle{remark}
\newtheorem{remark}{Remark}

\newcommand{\mL}{\mathcal L}

\newcommand{\mF}{\mathcal {F}}
\newcommand{\mP}{\mathcal {P}}
\newcommand{\vp}{\varphi}

\newcommand{\bpsi}{\bar\psi}
\newcommand{\bvp}{\bar\varphi}
\newcommand{\bm}{\tilde m}
\newcommand{\tpsi}{\tilde\psi}

\newcommand{\hD}{\hat D}
\newcommand{\hF}{\hat F}
\newcommand{\hR}{\hat R}
\newcommand{\Dom}{\Delta_\omega}

\newcommand{\bl}{\tilde\lambda}
\newcommand{\tom}{\tau^\omega}
\newcommand{\hDom}{\hat\Delta_\omega}

\newcommand{\Domo}{\Delta_{\omega, 0}}

\newcommand{\spltp}{\lfloor n^c\rfloor}

\newcommand{\N}{\mathbb{N}}

\newcommand{\p}{\mathbb{P}}
\newcommand{\Z}{\mathbb{Z}}

\newcommand{\amin}{{\underline \alpha}_0}
\newcommand{\amax}{{\underline \alpha}_1}

\newcommand{\eps}{\varepsilon}

\begin{document}
\title[Quenched decay of correlations]{Quenched decay of correlations for slowly mixing systems}
\author{Wael Bahsoun$^{\dagger}$}
\address{Department of Mathematical Sciences, Loughborough University,
Loughborough, Leicestershire, LE11 3TU, UK}
\email{$\dagger$ W.Bahsoun@lboro.ac.uk}
\email{$\ddagger$ M.Ruziboev@lboro.ac.uk}
\author{Christopher Bose$^{*}$}
\address{Department of Mathematics and Statistics, University of Victoria,
   PO BOX 3045 STN CSC, Victoria, B.C., V8W 3R4, Canada}
\email{$*$ cbose@uvic.ca}
\author{Marks Ruziboev$^{\ddagger}$}
\thanks{WB and MR would like to thank The Leverhulme Trust for supporting their research through the research grant RPG-2015-346. CB's research is supported by a research grant from the National Sciences and Engineering Research Council of Canada. The authors would like to thank V.\ Baladi for useful communications and helpful comments.}
\date{\today}

\begin{abstract}
We study random towers that are suitable to analyse the statistics of slowly mixing random systems. We obtain upper bounds on the rate of \emph{quenched} correlation decay in a general setting.  We apply our results to the random family of Liverani-Saussol-Vaienti maps with parameters in $[\alpha_0,\alpha_1]\subset (0,1)$ chosen independently with respect to a distribution $\nu$ on  $[\alpha_0,\alpha_1]$ and show that the \emph{quenched decay of correlation} is governed by the fastest mixing map in the family.  
In particular, we prove that for every $\delta >0$, for almost every $\omega \in [\alpha_0,\alpha_1]^\mathbb Z$, the upper bound $n^{1-\frac{1}{\alpha_0}+\delta}$ holds on the rate of decay of correlation for H\"older observables on the fibre over $\omega$.  
For three different distributions $\nu$ on $[\alpha_0,\alpha_1]$ (discrete, uniform, quadratic), we also derive \textit{sharp} asymptotics on the measure of return-time intervals for the quenched dynamics, ranging from $n^{-\frac{1}{\alpha_0}}$ to $(\log n)^{\frac{1}{\alpha_0}}\cdot n^{-\frac{1}{\alpha_0}}$ to $(\log n)^{\frac{2}{\alpha_0}}\cdot n^{-\frac{1}{\alpha_0}}$ respectively. 

 \end{abstract}

\maketitle
\pagestyle{myheadings} 
\tableofcontents

\section{Introduction}
In this paper we study statistical properties of systems that evolve according to deterministic laws driven by a random process. Such systems are called random dynamical systems and they are often studied via analysis of a related deterministic system, the skew product map $T:X\times\Omega\to X\times\Omega$ given by: 
$$T(x,\omega):= (f_\omega(x), \sigma\omega),$$
where $\{f_\omega\}_{\omega\in\Omega}$ is a family of transformations that map $X$, the phase space, into itself, and $\sigma$ is a measure preserving map on $\Omega$, the noise space. The $f_{\omega}$'s are often referred to as the fibre maps and $\sigma$ is called the base map or the driving system. The fibre maps are the deterministic components of the random system, while the base map invokes the required randomness, or time dependence, or parameter drift in the system.

Recently there has been a remarkable interest in studying statistical limit theorems for random dynamical systems \cite{ANV, ALS, BB, DGFV1, DGFV2, HNTV, HRY, K, NTV}. Most of these results assume some knowledge about the rate of correlation decay of the random system under consideration. In this work, we develop random towers that are suitable to study \emph{quenched}\footnote{Quenched results in random dynamical systems refer to pathwise results for almost every $\omega$.} correlation decay for slowly mixing random systems. We obtain a general result on the rate of \emph{quenched} correlation decay. Moreover, we apply our results to answer the following questions: in what way does an individual map $f_\omega$, or a group of $f_\omega$'s, dictate the rate of quenched\footnote{In a simple model, yet important in the study of intermittent transition to turbulence \cite{PM}, the first question was answered in \cite{BBD} only for the annealed dynamics; i.e., for the dynamics averaged over $\Omega$, and only for a specific distribution on $\Omega$. Precisely \cite{BBD} considered a system that has only two fibre maps and with the base system being a Bernoulli shift.} correlation decay of the random system? A second question is: how does the distribution on $\Omega$ (the measure preserved by $\sigma$) effect the quenched statistics of the system? We answer the above two questions in the framework of the Pomeau-Manneville family \cite{PM} using the version popularised by Liverani-Saussol-Vaienti \cite{LSV}. Such systems have attracted the attention of both mathematicians and physicists (see \cite{LS} for a recent work in this area). In particular, for Liverani-Saussol-Vaienti \added[id=WB]{(LSV)} maps with parameters in $[\alpha_0,\alpha_1]\subset (0,1)$ and base dynamics $([\alpha_0,\alpha_1]^{\mathbb{Z}}, \sigma, \nu)$ we show via a general random tower construction, that the \emph{quenched decay of correlation} is governed by the fastest mixing map. Precisely, we prove that $n^{1-\frac{1}{\alpha_0}+\delta}$ is an upper bound on the rate of quenched decay of correlation, for all $\delta>0$. To illustrate the role that $\delta>0$ plays in the quenched  decay rate, and to address the second question above, we also obtain \emph{sharp} asymptotics on the position of return time intervals for the quenched dynamics in the Liverani-Saussol-Vaienti family that depend on the randomising distribution. In particular, we show how different distributions on $[\alpha_0,\alpha_1]$ (discrete, uniform, quadratic) change the \textit{sharp} asymptotics on the position of return time intervals for the quenched dynamics from $n^{-\frac{1}{\alpha_0}}$ to $(\log n)^{\frac{1}{\alpha_0}}\cdot n^{-\frac{1}{\alpha_0}}$ to $(\log n)^{\frac{2}{\alpha_0}}\cdot n^{-\frac{1}{\alpha_0}}$.   

In Section \ref{sec:introduction} we recall standard definitions and notation from random dynamical systems and present various natural notions of correlation decay in this setting. In Section \ref{sec:setup} we build random towers for our system and detail the dynamical hypotheses that are in force throughout the paper.  Our main general results are contained in  Section \ref{sec:results} (Theorems \ref{thm:acsm} and \ref{thm:DC})  where we prove existence and correlation decay estimates respectively for the dynamics on the random towers. In Section \ref{sec:lsv_app} we present detailed computations applying our general results to the case of random LSV maps on the interval. Three different randomising distributions are 
investigated: discrete, uniform and quadratic. At the end of Section \ref{sec:lsv_app} we also compute exact asymptotics for the measure of the return sets on the base of the 
random towers.  In Section \ref{sec:existence} we prove Theorem \ref{thm:acsm}.  The expansion and distortion conditions and related estimates are the main tools used in this section. 
In the next section, Section \ref{sec:decay1}, we introduce random stopping times and derive asymptotics on their distributions in preparation for a coupling argument. In Section \ref{sec:decay2} we obtain decay of correlation estimates (upper bounds) for observables on our random towers. Both future and past decay estimates are derived. We conclude with Section \ref{sec:appendix} where we present some technical results that are used repeatedly in the paper. 
Notation: We use $a\lesssim b$ if there exists universal constant $C$ such that $a \le Cb$;  $\sim$, $o$, $O$ will have their usual meaning.

\section{Random dynamical systems}\label{sec:introduction}
Let $(\mathbb A, \mathcal F, p)$ be a Borel probability space, let $\Omega = \mathbb{A}^{\mathbb Z}$ be equipped with 
product measure $P:= p^{\mathbb Z}$ and let $\sigma:\Omega\to \Omega$ denote the $P-$preserving two-sided shift map. Let $(X, \mathcal B) $ be a measurable space. Suppose that  $f_u :X\to X$ is  a family of measurable maps defined for $p$-almost every $u \in \mathbb A$  such that  the {\it skew product} 
$$
T: X\times \Omega \to X\times \Omega, \,\, T(x, \omega)=(f_{[\omega]_0}(x), \sigma\,\omega)
$$
 is measurable with respect to $\mathcal B\times \mathcal  F$ were $[\omega]_k \in \mathbb A$ denotes the $k$-th coordinate of $\omega \in \Omega$.  In order to simplify notation, we will normally
 write $f_\omega := f_{[\omega]_0}$ when there is no danger of confusion. So, for example, $f_{\sigma \omega} = f_{[\omega]_1}$. The resulting  i.i.d.\ random 
 map associated to the family $\{f_\omega\}$ can be viewed as follows:
letting  $X_\omega:=X\times \{\omega\}$ denote the fiber over $\omega$ and $f^n_\omega= f_{\sigma^{n-1}\omega}\circ \cdots \circ f_\omega: X_\omega \to X_{\sigma^n \added[id=WB]{\omega}}$ we have $T^n(x, \omega)=(f^n_\omega(x), \sigma^n\,\omega)$. We say that $\mu$ is a $T$-invariant measure if $\mu(T^{-1}A)=\mu(A)$ for any $A\in \mathcal B   \times\mathcal F$. 
Assume that  $\{ X_\omega\}_{\omega\in\Omega}$  forms measurable partition\footnote{This is  satisfied, for example, when $\mathbb A$ is Hausdorff so that $\{\omega\}$ is closed.} of $X\times \Omega$. We are interested in $T$-invariant probability measures, $\mu$, such that $\pi_*\mu=P$, where $\pi$ is the projection onto $\Omega$. Then by Rokhlin's disintegration theorem (see \cite{ViOl} or  \cite{Rok} ), for any such measure $\mu$ there exists an (essentially unique) system of probability measures $\mu_\omega$ \added[id=WB]{on $X_\omega$} such that  for any $A\in\mathcal B  \times  \mathcal F$
\begin{equation}\label{disintegration}
\omega \mapsto \mu_\omega(A) \,\, \text{is measurable and} \,\, \mu(A)=\int \mu_\omega(A)dP(\omega).
\end{equation}
It is easy to check that 
 $\mu$ is $T$-invariant if and only if $(f_\omega)_\ast\mu_\omega=\mu_{\sigma\omega}$ for $P$- a.e. $\omega$, a property
we naturally refer to as \emph{$f_\omega-$equivariance} (or simply \emph{equivarariance}, when the random map is understood) of the family $\{\mu_\omega\}$.

In this paper we study statistical properties of the equivariant family of measures  $\{\mu_{\omega}\}$ for $P$-almost every $\omega\in\Omega$, when 
$\mu_\omega$ is absolutely continuous with respect to Lebesgue measure $m$ on $X$.
More precisely we study \emph{future} and \emph{past} \emph{quenched correlations}:  given  $\varphi, \psi: X\times \Omega \to \mathbb R$ define
 \emph{future} and \emph{past} fibre-wise correlations
 \begin{equation}\label{dycorr}
\begin{split}
Cor^{(f)}_{n, \omega}(\varphi, \psi)=
\int (\varphi\circ f^{n}_\omega)\psi d\mu_\omega-\int \varphi d\mu_{\sigma^n\omega}\int \psi d\mu_{\omega},\\
Cor^{(p)}_{n, \omega}(\varphi, \psi)=
\int (\varphi\circ f^{n}_{\sigma^{-n}\omega})\psi d\mu_{\sigma^{-n}\omega}-\int \varphi d\mu_{\omega}\int \psi d\mu_{\sigma^{-n}\omega}.
\end{split}
\end{equation}
\begin{definition}
 Let $\mathcal B_1$ and $\mathcal B_2$ be two Banach spaces on $X\times \Omega$ and let $\{\rho_n\}_{n\in\mathbb N}$ be a sequence of positive numbers such that $\lim_{n\to \infty}\rho_n=0$. We say that $f_\omega$ admits \emph{quenched decay of correlations} at rate $\rho_n$ if 
for $P$-almost all $\omega$ and for any  $\varphi\in \mathcal B_1$ and $\psi \in \mathcal B_2$ there are constants $C_\omega$ and $C_{\varphi, \psi}$ such that  
\begin{equation}\label{opcorr}
|Cor^{(f)}_{n, \omega}(\varphi, \psi)|\le C_\omega C_{\varphi, \psi}\rho_n, \quad
 |Cor^{(p)}_{n, \omega}(\varphi, \psi)|\le C_\omega C_{\varphi, \psi}\rho_n.
\end{equation}
\end{definition}
\begin{remark}
Note that if $C_\omega$ is $P$-integrable, then this implies the same rate for the integrated correlations; i.e., $\int_{\Omega}Cor^{(f)}_{n, \omega}(\varphi, \psi)dP\le \hat C_{\varphi, \psi}\rho_n$. The importance of knowing the rate of the integrated correlations is due to its relation to the annealed correlations of the skew product. Indeed, setting $\bar\varphi:=\int_{\Omega}\varphi d\mu_{\omega}$ and  $\bar\psi:=\int_{\Omega}\psi d\mu_{\omega}$ we have
$$Cor_{n}(T, \varphi, \psi)=\int_{\Omega}Cor^{(f)}_{n, \omega}(\varphi, \psi)dP+ Cor_{n}(\sigma,\bar\varphi, \bar\psi).$$
 \end{remark}

\section{Abstract tower setting}\label{sec:setup}
 
A main tool, in particular in the absence of spectral techniques, to study statistical properties of dynamical systems is the so called Young Tower \cite{Y98, Y99}. Young Towers \replaced[id=WB]{have}{has} been used extensively to obtain rates of decay of correlations for nonuniformly hyperbolic systems (see for example \cite{ALP, AP, G06, Mel} and references therein).  In this section we describe random towers which were first considered in \cite{BBMD} to study \emph{quenched} statistical properties of i.i.d. unimodal maps. \added[id=WB]{Later the work of \cite{BBMD} was extended in \cite{DZ} to cover a wider class of i.i.d. unimodal maps}. Building on ideas from \cite{BBMD, DZ} we study random towers with slowly decaying tails. Let  $\Lambda\subset X$  be a measurable set with $m(\Lambda)=1$. Consider  a  family of maps $f_\omega:X\to X$, where $f_\omega$ depends only on  zeroth  coordinate of $\omega$. We say that $f_\omega$ admits a random tower on $\Lambda\subset X$  if for almost every $\omega\in\Omega$  there exists a countable partition $\{\Lambda_j(\omega)\}_j$ of $\Lambda$ and  a return time function $R_\omega:\Lambda\to \mathbb N$ that is constant on each $\Lambda_j(\omega)$ such that  $f^{R_\omega}_\omega(x)=f_{\sigma^{R_\omega(x)-1}\omega}\circ \cdots\circ f_{\sigma\omega}\circ f_\omega (x)\in\Lambda$ for $P$-almost every $\omega\in\Omega$ and $m$ -almost every $x\in\Lambda.$ 
Given the above information we define a random  tower for almost every $\omega$ as
\begin{equation}\label{tower}
\Delta_\omega=\left\{(x, \ell)\in\Lambda\times\mathbb {Z}_+\mid
x\in \cup_j\Lambda_j(\sigma^{-\ell}\omega),\ell\in\mathbb Z_+, 0\le \ell\le R_{\sigma^{-\ell}\omega}(x)-1\right \}
\end{equation} 
and random tower map $F_\omega: \Delta_\omega\to \Delta_{\sigma\omega}$ by 
\begin{equation}
F_\omega(x, \ell)=\begin{cases} 
(x, \ell+1), \quad\text{if}\quad \ell+1<R_{\sigma^{-\ell}\omega}(x) \\
(f^{\ell+1}_{\sigma^{-\ell}\omega}, 0), \quad\text{if}\quad \ell+1=R_{\sigma^{-\ell}\omega}(x).
\end{cases}
\end{equation}
Denote by $\Delta_{\omega, \ell}:=\{(x, \ell)\in\Delta_\omega\}$ the $\ell$th level of the tower, which is a copy of $\{x\in\Lambda\mid R_{\sigma^{-\ell}\omega}(x)>\ell\}$; for instance  $\Delta_{\omega, 0}=\Lambda$ and $F^{R_\omega}_\omega|\Delta_{\omega, 0}=f^{R_\omega}_\omega|\Lambda$. Let $\Delta=\{\Delta_{\omega}\}_{\omega\in\Omega}$.  Then   $F=\{F_\omega\}_{\omega\in\Omega}$ is a fibered system on $\Delta$; see Figure \ref{fig:pic1} for a pictorial representation.

\begin{figure}[!ht]
	\centering{
		\includegraphics[scale=0.4]{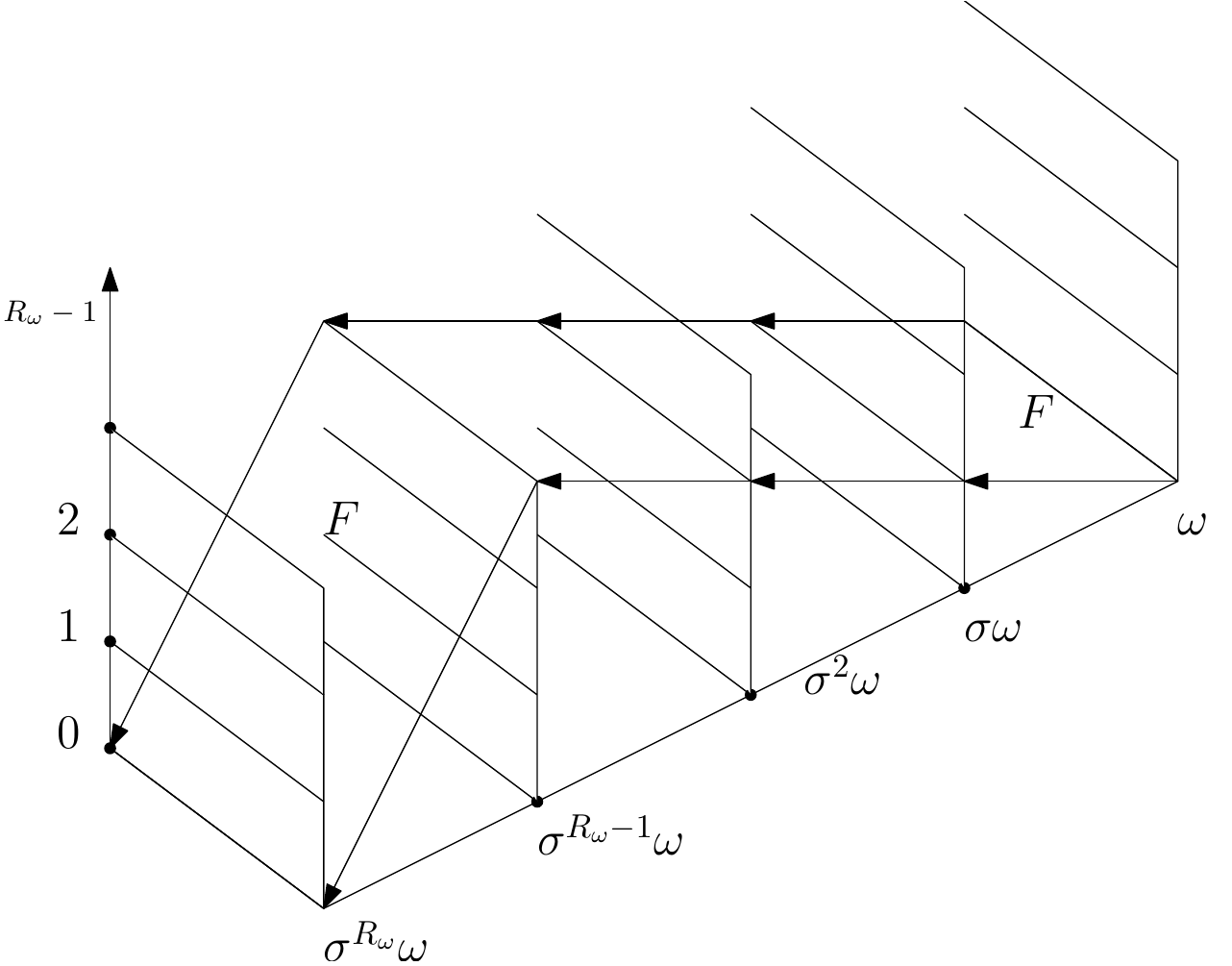}
	\caption{Dynamics on the tower}\label{fig:pic1}}
\end{figure}

Notice that $\{\Lambda_j(\omega)\}_j$ induces a countable partition $\mathcal P_\omega$ on each $\Delta_\omega$:
$$
\mathcal P_\omega:=\left\{ F^\ell_{\added[id=WB]{\sigma^{-\ell}}\omega}(\Lambda_j(\added[id=WB]{\sigma^{-\ell}}\omega))\mid R_\omega|\Lambda_j(\sigma^{-\ell}\omega)\ge \ell+1, \ell\in\mathbb Z_+\right\}.
$$
For $(x,\ell) \in \Delta_\omega$, let $\hR_\omega$ denote the first return time to the base of the tower $\Delta_{\sigma^{\hR_\omega} \omega}$ i.e.
\begin{equation}\label{eq:hatR}
\hR_\omega(x, \ell)=R_{\sigma^{-\ell}\omega}(x)-\ell.
\end{equation}
The reference measure $m$ and $\sigma$-algebra on $\Lambda$ naturally lifts to $\Delta_\omega$ and by abuse of notation we call it $m$. The lifted $\sigma$-algebra will be denoted by $\mathcal B_\omega$. 
Next we define the separation time $s:\Delta\times\Delta\to \mathbb Z_+\cup\{\infty\}$ for almost every $\omega$ by setting $s(z_1, z_2) =0$ {if} $ z_1$ {and} $z_2$  {lie in different towers} $\Dom$ and if $ z_1, z_2\in\Dom$ then 
$$
s(z_1, z_2)= \min\{n\ge 0\mid (F^{R_\omega}_\omega)^n(z_1)\,\,\text{and}\,\, (F^{R_\omega}_\omega)^n(z_2)\,\, \text{lie in distinct elements of}\,\, \mathcal P_\omega, \},
$$
Below we refer to $\Lambda$ as the zeroth level of the tower. 
We assume that the random tower satisfies the following properties. 
\begin{itemize}
\item[(P1)] \textbf{Markov}: for each $\Lambda_j(\omega)$ the map $F^{R_\omega}_\omega|\Lambda_j(\omega):\Lambda_j(\omega)\to \Lambda$ is a bijection;
\item[(P2)] \textbf{Bounded distortion}: There are constants $D>0$ and $0<\gamma<1$ such that for all $\omega$ and each $\Lambda_j(\omega)$ the map $F^{R_\omega}_\omega|\Lambda_j(\omega)$ and its inverse are non-singular with respect to $m$ with corresponding Jacobian $JF^{R_\omega}_\omega|\Lambda_j(\omega)$ which is positive and for each $x,y\in \Lambda_j(\omega)$  satisfies the following
\begin{equation}\label{bddd}
\left|\frac{JF^{R_\omega}_\omega(x)}{JF^{R_\omega}_\omega(y)}-1\right|\le D\gamma^{s(F^{R_\omega}_\omega(x, 0), F^{R_\omega}_\omega(\replaced[id=WB]{y}{x}, 0))};
\end{equation}
\item[(P3)] \textbf{Weak expansion}: $\mathcal P_\omega$ is a generating partition for $F_\omega$ i.e. diameters of the partitions $\vee_{j=0}^n F^{-j}_{\omega}\mathcal P_{\sigma^{j}\omega}$ converge to zero as $n$ tends to infinity;
\item[(P4)] \textbf{Return time asymptotics}: There are constants $C>0$,  \added[id=WB]{$a>1$}, $b\ge 0$, $u>0$, $v>0$, a full measure subset $\Omega_1\subset \Omega$ and a random variable $n_1:\Omega_1\to \mathbb N$ such that
 \begin{equation}
\begin{cases} \label{tail}
m\{x\in\Lambda\mid R_\omega(x)>n\}\le C\frac{(\log n)^b}{n^a}, \,\,\text{whenever} \,\,n\ge n_1(\omega),\\
P\{n_1(\omega)>n\} \le C e^{-un^v};
\end{cases}
\end{equation} 
\item[(P5)] \textbf{Aperiodicity}: There are $N\in\mathbb N$ and 
$\{t_i\in\mathbb Z_+\mid i=1, 2, ..., N\}$   such that g.c.d.$\{t_i\}=1$ and $\epsilon_i >0$ so that
for almost every $\omega \in \Omega$ and $i = 1, 2, \dots N$ we have
$m\{x\in\Lambda\mid R_\omega(x)=t_i\}>\epsilon_i$.
\item[(P6)] \textbf{Finiteness}  There exists an $M>0$ such that   $m(\Dom)\le M$ for all $\omega\in \Omega$.
\item[ (P7)] \textbf{Annealed return time asymptotics}: There are constants $C>0$,  $\hat b\ge 0$ and $a>1$ such that $(P\times m)\{x\in\Lambda |\, R_\omega=n\}\le C\frac{(\log n)^{\hat b}}{n^{a+1}}$.
\end{itemize}

\subsection{Tower projections}  For almost every $\omega\in \Omega$ and $(x, \ell)\in \Dom $ we define tower projections $\pi_\omega:\Dom\to X_{\added[id=WB]{\omega}}$  as $\pi_\omega(x, \ell)= f^\ell_{\sigma^{-\ell}\omega}(x)$. Then $\pi_\omega$ is a semi-conjugacy  i.e. $\pi_\omega\circ F_\omega =f_\omega\circ\pi_\omega$. Indeed, for $(x, \ell)\in \Dom$ we have $f_\omega(\pi_\omega(x, \ell))=f_\omega\circ f^{\ell}_{\sigma^{-\ell}\omega}(x)$. 
On the other hand, \added[id=WB]{since $F(x,\ell)\in\Delta_{\sigma\omega}$}  
$$
\pi_\omega(F_\omega(x, \ell))=\begin{cases} 
\pi_\omega(x, \ell+1) \quad \text{if} \quad R_{\sigma^{-\ell}\omega}(x)> \ell+1\\
\pi_\omega(f^{\ell+1}_{\sigma^{-\ell}\omega}(x), 0) \quad \text{otherwise} 
 \end{cases}
 =f^{\ell+1}_{\sigma^{-\ell}\omega}(x).
$$
Now, if $\nu_\omega$ is an absolutely continuous family of $F_\omega$-equivariant probability measures  on $\Dom$,   then $\mu_\omega:=(\pi_\omega)_\ast\nu_\omega$ is a family of  $f_\omega$-  equivariant  probability measures on $X\times\Omega$.  Since each $f_\omega$ is nonsingular, if $A\subset X$ is such that $m(A)=0$ then $m(\pi_\omega ^{-1}(A))=0$,  which implies  $\nu_\omega(\pi_\omega ^{-1}(A))=0$, consequently $\mu_\omega(A)=0$.  Therefore each $\mu_\omega$ is absolutely continuous. 
\section{Statement of main results}\label{sec:results}
In this section we state general theorems concerning quenched correlation decay for slowly mixing systems. We start this section by introducing some function spaces on $\Delta$, which are necessary to state the theorems. These spaces appeared in the present form in   \cite{BBMD}.
Below we let constants  $u>0,$ $v>0$, $a>1$, $b\ge 0$, $\gamma< 1$ be as in (P2) and  (P4) above and set   
\begin{equation*}\begin{aligned}
\mF^+_\gamma=\{\varphi_\omega:\Dom \to \mathbb R\mid & \exists C_\varphi>0, \forall I_\omega\in \mP_\omega, \,\,\, 
\text{either} \,\, \varphi_\omega|{I_\omega}\equiv 0 \\
&\text{or} \,\,\, \varphi_\omega | {I_\omega}> 0 \,\,\, \text{and} \,\,\, \left| \log \frac{\varphi_\omega(x)}{\varphi_\omega(y)}\right|\le C_\varphi \gamma^{s(x, y)}, \forall x, y\in I_\omega\}.
\end{aligned}
\end{equation*}
Let  $K_\omega: \Omega \to \mathbb R_+$ be a random variable with $\inf_{\Omega} K_\omega >0$ and 
\begin{equation}\label{Komega}
P\{\omega\mid K_\omega >n\} \le e^{-un^v}.
\end{equation}
Define the space of random  bounded  functions as 
\begin{equation*}\begin{aligned}
\mL^{K_\omega}_\infty=\{\vp_\omega:\Dom \to \mathbb R\mid \exists C^\prime_\vp>0,  \sup_{x\in\Dom} |\vp_\omega(x)| \le C_\vp' K_\omega \}
\end{aligned}
\end{equation*}
and a space of random Lipschitz  functions 
\begin{equation*}\begin{aligned}
\mF^{K_\omega}_\gamma=\{\vp_\omega\in \mL^{K_\omega}_\infty \mid \exists C_\vp>0,   |\vp_\omega(x)-\vp_\omega(y)| \le C_\vp K_\omega \gamma^{s(x, y)}, \,\, \forall  x, y \in \Dom\}.
\end{aligned}
\end{equation*}
Finally  we let $\mathcal B$ be a $\sigma$-algebra on $\Delta$ defined as follows: $B\in\mathcal B$ if and only if for each $\omega$ the intersection $B_\omega=B\cap \Dom \in \mathcal B_\omega$. Let  $\{\nu_\omega\}_{\omega\in\Omega}$ be a fibered equivariant family of measure i.e. $(F_\omega)_\ast\nu_\omega=\nu_{\sigma\omega}$. Let $\mu_\omega =(\pi_\omega)_\ast\nu_\omega$ and $\mu (A)=\int_{\Omega}\mu_\omega(A_\omega)dP(\omega)$. Then $\mu$ is  $T$-invariant (for the skew product $T$ on $X\times\Omega$).  We say that  $\nu$ is exact/mixing for $F$ if $\mu$ is exact/mixing for the skew product $T$. We can formulate equivalent conditions as follows. 
\begin{definition}
\text{ }
\begin{itemize}
\item[(i)] The fibered system $(F, \nu)=(F_\omega, \nu_\omega)_{\omega\in\Omega}$ is exact iff $\vee_{n=0}^{\infty}F^{-n}\mathcal B$ is trivial; i.e., \added[id=WB]{for any $B\in \vee_{n=0}^{\infty}F^{-n}\mathcal B$, either for almost all $\omega$, $\nu_{\omega}(B)=0$ or for almost all $\omega$, $\nu_\omega(B)=1$.} 
\item[(ii)] The random skew product $(F, \nu)$ is mixing  iff for all $\varphi, \psi \in L^2(\nu)$, 
$$
\lim_{n\to \infty}\left|\int_\Omega\int_{\Dom}\varphi_{\sigma^n\omega}\circ F^n_\omega\cdot\psi_\omega d\nu_\omega dP -\int_\Omega\int_{\Dom}\varphi_{\omega}d\nu_\omega dP\,\int_\Omega\int_{\Dom}\psi_{\omega}d\nu_\omega dP\right| =0.
$$
\end{itemize}
\end{definition}
We note that in our situation exactness implies mixing (\cite{BBMD}, section 4). The first result is the existence of absolutely continuous sample measures. 
\begin{theorem}\label{thm:acsm} Let $(F, \Delta)$ be the fibered system described above. 
There exists an
$F$-equivariant family of absolutely continuous sample probability measure $\nu_\omega =h_\omega m$
defined for almost every $\omega \in \Omega$, which is exact,  and hence mixing. Moreover,  there exists $K_\omega$ satisfying \eqref{Komega} such that $h_\omega\in \mF_{\gamma}^{K_\omega}\cap \mF_\gamma^+$ for almost every $\omega\in\Omega$.  
\end{theorem}
 The next main result is about the decay of future and past quenched correlations.
\begin{theorem}\label{thm:DC}
Let $\delta>0$. Let $K_\omega$ \replaced[id=WB]{be the function given in Theorem \ref{thm:acsm}}{satisfy \eqref{Komega}}. There exits a full measure set $\Omega_0\subset\Omega$ and a random variable $C_\omega$ on $\Omega_0$ such that for every $\vp, \psi \in \mF_\gamma^{K_\omega}$  there exits a constant $C_{\vp, \psi}$ such that  for every $\omega\in \Omega_0$ 
\begin{itemize}
\item[(i)]{``Future" operational correlations} :  
$$
\left| \int (\varphi_{\sigma^n\omega}\circ F^{n}_\omega)\psi_\omega dm-\int \varphi_{\sigma^n\omega} d\mu_{\sigma^n\omega}\int \psi_\omega dm\right| \le C_\omega C_{\varphi, \psi}{n^{1+\delta-a}}.
$$
\item[(ii)]{``Past" operational correlations} :  
$$
\left| \int (\varphi_{\omega}\circ F^{n}_{\sigma^{-n}\omega})\psi_{\sigma^{-n}\omega} dm-\int \varphi_{\omega} d\mu_{\omega}\int \psi_{\sigma^{-n}\omega}dm\right| \le  C_\omega C_{\varphi, \psi}{n^{1+\delta-a}}.
$$
\end{itemize}
Moreover, there exist constants $C>0$, $u'>0$, $v'\in (0, 1)$ such that  
$$
P\{C_\omega> n\}\leq C e^{-u'n^{v'}}.
$$
\end{theorem}
\begin{remark}\label{justificationrate}
A quenched correlation decay rate of the form $\frac{(\log n)^b}{n^{a-1}}$, which is analogous to what one expects in the deterministic setting, cannot be achieved since we want to get information on the integrability of the $C_{\omega}$ in Theorem \ref{thm:DC}. The shift of the Lipschitz constant $K_\omega$, and hence the dependence of that constant on $n$, in equation \eqref{eq:corr_tpsi} and the non-uniformity of the tail in (P4) are the main reasons for getting a rate at the order $\frac{1}{n^{a-1+\delta}}$, for any $\delta>0$. See Footnote \ref{justfood} for more details.
\end{remark}
\section{Applications to random LSV maps}\label{sec:lsv_app}
In this section we illustrate our results with applications to the family of intermittent LSV maps as 
described in \cite{LSV}. Let $0< \alpha <1$  and consider $f_\alpha:I \to I$ defined as 
\begin{equation}\label{def:lsv}
f_\alpha (x)= \begin{cases} 
x(1+2^\alpha x^\alpha), \quad x\in \left[ 0, \frac{1}{2}\right], \\
2x-1, \quad   x\in \left( \frac{1}{2}, 1 \right].
\end{cases} 
\end{equation}
To define a random LSV map we fix two positive numbers $0< \alpha_0 < \alpha_1 <1$ and  let $\nu$ be a probability measure on $[\alpha_0, \alpha_1]$.  Set  $\Omega=[\alpha_0, \alpha_1]^{\mathbb Z}$ and $P=\nu^{\mathbb Z}$. Then the shift  map $\sigma:\Omega\to \Omega$ preserves $P$.  Let $\alpha:\Omega \to [\alpha_0, \alpha_1]$ be the projection to the zeroth coordinate and let $f_{\alpha(\omega)}=f_\omega$. We consider the skew product  $T: I\times \Omega \to I\times \Omega$ defined by
$$
T(x, \omega)=(f_\omega(x), \sigma(\omega)).
$$
Compositions of $f_\omega$ are given by $f^n_\omega=f_{\sigma^{n-1}\omega}\circ \cdots \circ f_{\sigma\omega}\circ f_{\omega}$.

For each $\omega$ we define a  sequence of pre-images of $\frac{1}{2}$ as follows. 
Let  $x_1(\omega)=\frac{1}{2}$,  and
\begin{equation}\label{eq:x_n}
x_n(\omega) =\left(f_\omega |_{(0, 1/2]}\right)^{-1}x_{n-1}(\sigma\omega) \,\, \text{for} \,\, n\ge 2.
\end{equation}  
Further let 
\begin{equation}\label{eq:x_n'}
x_0'(\omega)=1, x_1'(\omega)=\frac{3}{4},  \,\,\, \text{and} \,\,\, x_{n}'(\omega)=\frac{x_n(\sigma\omega)+1}{2} \,\,\, \text{for}\,\,\, n\ge 2.
\end{equation}
The sequences $\{x_n(\omega)\}$ and $\{x_n'(\omega)\}$ will allow us to define the 
random tower structure. First of all notice that from the definition of $x_n'(\omega)$ we have $f_{\omega}(x_n'(\omega))=x_n(\sigma\omega)$, $f_{\omega}^2(x_n'(\omega))=x_{n-1}(\sigma^2\omega), ...,$
$f_{\omega}^n(x_n'(\omega))=x_1(\sigma^n\omega)=\frac{1}{2}$.
Let $\Lambda:=\left(\frac{1}{2}, 1\right]$. The sequence $\{x_n'(\omega)\}_{n\ge 0}$ generates a partition $\mP_{\omega}=\{(x_n'(\omega), x_{n-1}'(\omega)]\mid n\ge 0\}$ on each $\Lambda\times\{\omega\}=\Domo$. Define the return time $R_\omega:\Domo\to \mathbb N$ by setting  
\begin{equation}\label{eq:lsv_return}
R_\omega|_{(x_n'(\omega), x_{n-1}'(\omega)]}=n.
\end{equation}

A fibered system is obtained by defining a tower $\Dom$ over each $\omega \in \Omega$ by
$$
\Dom=\cup_{\ell=0}^\infty \cup_{i=\ell +1}^{\infty}(x_{i}'(\sigma^{-\ell}\omega), x_{i-1}'(\sigma^{-\ell}\omega)]\times {\ell}.
$$ 
The fibered map $F:(\omega, \Dom) \to (\sigma \omega, \Delta_{\sigma \omega}) $  from equation (\ref{tower}) can be expressed in this notation as follows:
Let $(x,\ell) \in \Dom$ over $\omega$, with $x \in (x_{i}'(\sigma^{-\ell}\omega), x_{i-1}'(\sigma^{-\ell}\omega)]$. There are two cases.  If $i > \ell +1$, then $F(x, \ell)= (x, \ell +1) \in \Delta_{\sigma \omega}$ since $x \in (x_{i}'(\sigma^{-\ell}\omega), x_{i-1}'(\sigma^{-\ell}\omega)] =  (x_{i}'(\sigma^{-\ell-1}\sigma \omega), x_{i-1}'(\sigma^{-\ell-1} \sigma \omega)] $.  On the other hand, if $i = \ell +1$ then 
$x \in (x_{\ell + 1}'(\sigma^{-\ell}\omega), x_{\ell}'(\sigma^{-\ell}\omega)]$, where
$R_{\sigma^{-\ell}\omega} \equiv \ell +1$, so the interval is 
mapped bijectively to  $(\frac{1}{2}, 1]$ by $f_{\sigma^{-\ell}\omega}^{\ell +1}$. Therefore in this case we have  $F(x, \ell) = ( f_{\sigma^{-\ell}\omega}^{\ell +1}(x), 0)$. 
\begin{proposition}\label{prop:applic_general}
The fibered system \replaced[id=CB]{$\{\Dom\}_{\omega \in \Omega}$}{$(\Dom)_{\omega \in \Omega}$} with fibered map $F$ defined above 
satisfies properties (P1)-(P3) and (P5).  In particular, the distortion condition (P3) is satisfied for any $\gamma\in[ \frac{1}{2},1)$ and $D< \infty$.  
\end{proposition}
\begin{proof}
Since every map in the family $f_\alpha$ expands by at least a factor of $2$ on return to the base interval $(\frac{1}{2}, 1]$, with full returns,  we see that the Markov and weak expansion properties are satisfied.  Furthermore, for each $\omega$, 
$\{ x \in \Lambda ~|~ R_\omega(x) = 1\} = (\frac{3}{4}, 1]$, which implies that the aperiodicity condition (P5) is satisfied. Since every $f_\alpha$ has negative Schwarzian derivative and this property is preserved under 
composition, we obtain the bounded distortion condition (P3) using the Koebe principle.  See Lemmas 4.8 and 4.10 in  \cite{BBD} for computations related to Schwarzian derivatives and \cite{MV} for more details about the use of the Koebe principle.  This completes the proof. 
\end{proof}

It remains to establish appropriate estimates on the return time asymptotics as in (P4) \added[id=CB]{and (P7)} and the uniform bound in (P6).
Observe, in view of the return-time formula \eqref{eq:lsv_return}, that
\begin{equation}\label{eq:R_for_xn} 
m\{ R_{ \omega} > \ell\} = \frac{1}{2} x_\ell( \sigma \omega).
\end{equation}

We will estimate terms on the right hand side of this expression.  
Let $\amin$ (respectively $\amax$) denote the special sequences of all $\omega_k = \alpha_0$ (respectively, all $\omega_k = \alpha_1$).  
Following Lemma 4.4 in \cite{BBD}, we obtain coarse estimates on the location of 
the $x_n(\omega)$. Translated into our setting these estimates imply,
for every $\ell, n\in \mathbb N$,
\begin{equation}\label{changeq}
x_\ell(\amin)\le x_\ell(\omega) \le  x_\ell(\amax).
\end{equation}
It is well known (see \cite{Y99} section 6, for example) that 
$x_\ell(\amin) \sim 
\frac{1}{2} \alpha_0^{-\frac{1}{\alpha_0}}\ell^{-\frac{1}{\alpha_0}}$ so if we define
$c_\ell(\amin) := x_\ell(\amin)\ell^{\frac{1}{\alpha_0}}$ and write
$x_\ell(\amin)=\frac{ c_\ell(\amin)}{\ell^{\frac{1}{\alpha_0}}}$ then 
$\lim_\ell c_\ell(\amin) = \frac{1}{2} \alpha_0^{-\frac{1}{\alpha_0}}:= c(\amin)$.
We define $c_\ell(\amax)$ and $c(\amax)$ analogously using $x_\ell(\amax)$.
Since, $m\{ R_{ \omega} > \ell\} =\frac{1}{2} x_\ell( \added[id=WB]{\sigma}\omega) \le  \frac{1}{2} x_\ell(\amax) \le C\ell^{-\frac{1}{\alpha_1}}$ this implies $m(\Dom)\le M$ for some $M>0$ independent of $\omega$. This proves (P6).  

We now check that  assumption (P4) is satisfied. 
Using definition \eqref{def:lsv} and 
the estimate $(1 + x)^{-\alpha} \leq 1 - \alpha x + \frac{\alpha(1+\alpha)}{2}x^2$, valid for  
$\alpha>0,~ x \geq 0$, and by substituting $x = [2x_n(\omega)]^{\alpha(\omega)}$ we obtain
\begin{equation}\label{eq:onestep}\begin{split}
\frac{1}{[x_{n}(\omega)]^{\alpha_0}} -\frac{1}{[x_{n-1}(\sigma \omega)]^{\alpha_0}} &\geq 
\alpha_0 2^{\alpha_0} [2x_n(\omega)]^{\alpha(\omega) - \alpha_0}\\
&- \frac{\alpha_0(1+\alpha_0)}{2}2^{\alpha_0}
[2x_n(\omega)]^{2\alpha(\omega) - \alpha_0}.\\
\end{split}
\end{equation}
Iterating this one-step estimate along the sequence  $x_k(\sigma^{\ell-k} \omega)$, $k= 1, 2, \dots \ell$  we obtain 
\begin{equation}\begin{split}\label{eq:xnsigma_n_1}
\frac{1}{[x_{\ell}(\omega)]^{\alpha_0}} \ge 2^{\alpha_0} &+ \alpha _02^{\alpha_0} \{
\sum_{k=2}^\ell[2x_k(\sigma^{\ell-k}\omega)]^{\alpha(\sigma^{\ell-k}\omega) - \alpha_0} \\ 
&- \frac{1+ \alpha_0 }{2} \sum_{k=2}^\ell [2x_k(\sigma^{\ell-k}\omega)]^{2\alpha(\sigma^{\ell-k}\omega) - \alpha_0}\}.
\end{split}
\end{equation}

Combining  equations \eqref{eq:xnsigma_n_1} and \eqref{changeq} implies that, for any 
parameter $q\geq 0$ 
\begin{equation}\label{xnomega}\begin{split}
&\frac{(\log\ell)^q}{\ell[x_{\ell}(\omega)]^{\alpha_0}} \ge \frac{2^{\alpha_0}(\log\ell)^q}{\ell} + 
\alpha_0 2^{\alpha_0} 
\biggl\{\frac{(\log\ell)^q}{\ell}\sum_{k=2}^\ell\left[\frac{2c_k(\alpha_0)}{k^{\frac{1}{\alpha_0}}}\right]^{\alpha(\sigma^{\ell-k}\omega) - \alpha_0}\\ 
&- \frac{1+ \alpha_0}{2} \frac{(\log\ell)^q}{\ell-1} \sum_{k=2}^\ell \left[\frac{2c_k(\alpha_1)}{k^{\frac{1}{\alpha_1}}}\right]^{2\alpha(\sigma^{\ell-k}\omega) - \alpha_0}\biggr\}\\
&\geq 
\alpha_0 2^{\alpha_0} \frac{(\log\ell)^q}{\ell}\biggl\{\sum_{k=2}^\ell
\left[\frac{2c_k(\alpha_0)}{k^{\frac{1}{\alpha_0}}}\right]^{\alpha(\sigma^{\ell-k}\omega) - \alpha_0}\\
&\hskip 2 cm - \frac{1+ \alpha_0}{2}  \left[\frac{2c_k(\alpha_1)}{k^{\frac{1}{\alpha_1}}}\right]^{2\alpha(\sigma^{\ell-k}\omega) - \alpha_0}\biggr\}= \frac{(\log\ell)^q}{\ell}\sum_{k=1}^\ell  A_k(\omega),
\end{split}
\end{equation}
where we have introduced notation $A_k(\omega)$ for the sequence of independent random variables
$A_1 \equiv 0$ and for $k=2, 3, \dots \ell$
\begin{equation}\label{eq:Ak}
A_k(\omega):= \alpha_02^{\alpha_0} \left[\frac{2c_k(\alpha_0)}{k^{\frac{1}{\alpha_0}}}\right]^{\alpha(\sigma^{\ell-k}\omega) - \alpha_0}  - \frac{1+ \alpha_0}{2} \alpha_02^{\alpha_0}\left[\frac{2c_k(\alpha_1)}{k^{\frac{1}{\alpha_1}}}\right]^{2\alpha(\sigma^{\ell-k}\omega) - \alpha_0}.
\end{equation}
From now on we write $A_k:=A_k(\omega)$. Note that $-\alpha_02^{\alpha_0} \leq A_k \leq \alpha_02^{\alpha_0}$ for every $k$. 

\bigskip
\noindent {\bf Assumption (A1) (Asymptotics on expectations)}\\
  Assume\footnote{Below in the examples, we will show that the assumption is satisfied for different types of distributions. In particular, we will show how different measures $\nu$ lead to different tail asymptotics.} there are 
constants $q=q(\nu)\geq 0$ and a constant $c(\nu)>0$ such that the following holds:\\
\begin{equation}\label{eq:Ak_asympt}
\frac{(\log\ell)^q}{\ell}\sum_{k=1}^{\ell}E_\nu A_k\to c(\nu).
\end{equation}
\noindent
Fix any $0< c <c(\nu)$.  Pick $N_1$ so that 
for all $\ell > N_1$, 
$$\frac{(\log\ell)^q}{\ell}\sum_{k=1}^{\ell}E_\nu A_k  \geq  \frac{c + c(\nu)}{2}.$$
Note that, given expression \eqref{eq:Ak_asympt},  $N_1$ depends only on the choice of $c$, and in 
particular is independent of $\omega$.  Set $r_0 = \alpha_02^{\alpha_0}$. 

\begin{lemma}\label{lemma:hoeffding}
For each $t>0$ we have 
\begin{equation}\label{eq:large_deviations}
P\left\{\frac{(\log \ell)^q}{\ell} \left|\sum_{k=1}^\ell A_k - \sum_{k=1}^\ell E_\nu A_k\right| \geq t \right\} \leq 
 \exp\left[-\frac{\ell t^2}{2r_0 (\log \ell)^{2q}}\right].
\end{equation}
\end{lemma}
\begin{proof} 
We may apply the classical result of Hoeffding (see \cite{Ho} Theorem 1,  or \cite{Mc} Lemma 1.2) to the sequence of independent random variables $A_k$, noting that instead of the bound 
$0 \leq A_k \leq 1$ we have $-r_0 \leq A_k \leq r_0$, accounting for the extra factor in the exponential. 
\end{proof}

Next, we apply the previous lemma to obtain a large deviation estimate on the normalized 
sums of $A_k$. For each $\ell > N_1$:

\begin{equation}\label{eqn:P4_derive}\begin{split}
P\left\{\frac{(\log \ell)^q}{\ell} \sum_{k=1}^\ell A_k < c \right\}
&=P\left\{\frac{(\log \ell)^q}{\ell} \{\sum_{k=1}^\ell  A_k - \sum_{k=1}^\ell  E_\nu A_k\} < c - \frac{(\log \ell)^q}{\ell}  \sum_{k=1}^\ell E_\nu A_k \right\}\\
&\leq P\left\{ \frac{(\log \ell)^q}{\ell} \left| \sum_{k=1}^\ell  A_k - \sum_{k=1}^\ell  E_\nu A_k \right| > \frac{c(\nu) - c}{2}\right\}\\
&\leq \exp\left[-\frac{\ell [c(\nu) - c]^2}{8r_0 (\log \ell)^{2q}}\right],
\end{split}
\end{equation}
where in the last line we have used equation (\ref{eq:large_deviations}).  Now define 
$$n_1(\omega) := \inf\{ n > N_1 ~|~ \forall  \ell > n, ~   \frac{(\log \ell)^q}{\ell} \sum_{k=1}^\ell A_k \geq c \}.$$

If $\ell > n_1(\omega)$ then equation (\ref{xnomega}) implies that 
\begin{equation}
 \frac{(\log\ell)^q}{\ell[x_{\ell}(\omega)]^{\alpha_0}}  \geq c
\end{equation}
and hence 
\begin{equation}\label{eq:upperxnomega}
x_{\ell}(\omega) \leq \biggl[ \frac{1}{c}\biggr]^{\frac{1}{\alpha_0}}
\frac{[\log \ell]^{\frac{q}{\alpha_0}}}{\ell^{\frac{1}{\alpha_0}}} \leq 2c(\nu)^{-\frac{1}{\alpha_0}}
\frac{[\log \ell]^{\frac{q}{\alpha_0}}}{\ell^{\frac{1}{\alpha_0}}},
\end{equation}
provided $\frac{c(\nu)}{2^{\alpha_0}} < c < c(\nu)$.  We may take $b:= \frac{q}{\alpha_0}$ and $a := \frac{1}{\alpha_0}$ and $C=c(\nu)^{-\frac{1}{\alpha_0}}$ in (P4).  Note that the constant $C$ is independent of $n$ and $\omega$.  Finally, we estimate, for $n > N_1$ and fixed $0<v<1$
\begin{equation} \begin{split}\label{eq:n1}
P\left\{ n_1( \omega) > n  \right\}
&\leq \sum_{\ell > n} \exp\left[-\frac{\ell [c(\nu) - c]^2}{8r_0 (\log \ell)^{2q}}\right]\\
&\leq \sum_{\ell > n} \exp\left[-u \ell^v\right] \leq C\exp\left[-u n^v\right]
\end{split}
\end{equation}
for suitable constants $u>0$ and $C < \infty$.  We can remove the restriction $n> N_1$ by substitution of a larger constant $C$ in the final expression
$$P\left\{ n_1( \omega) > n  \right\} \leq Ce^{-un^v}, $$
completing the second condition in (P4). 
Finally we verify (P7). 

First, we have
\begin{equation}\label{step1}
\begin{split}
(P\times m)\{x\in\Lambda |\, R_\omega=n\}&=(P\times m)\{x\in\Lambda |\, R_\omega>n-1\}\\
&-(P\times m)\{x\in\Lambda |\, R_\omega>n\}\\
&=\frac12 E_\nu [x_{n-1}(\omega)]-\frac12 E_\nu [x_{n}(\omega)]\\
&=\frac12 E_\nu [x_{n-1}(\sigma\omega)]-\frac12 E_\nu [x_{n}(\omega)]\\
&=\frac12 E_\nu[x_{n-1}(\sigma\omega)-x_{n}(\omega)],
\end{split}
\end{equation}
where we have used the fact that $P$ is $\sigma$-invariant to write $E_\nu [x_{n-1}(\omega)]=E_\nu [x_{n-1}(\sigma\omega)]$.
Now recall that $f_\omega(x_n(\omega))=x_{n-1}(\sigma\omega)$ and $f_\omega(x_n(\omega))= x_n(\omega)+2^{\alpha(\omega)} (x_n(\omega))^{\alpha(\omega)+1}$. Using this fact, \eqref{step1}, \eqref{eq:upperxnomega} and $0\le 2x_n(\omega)\le 1$, we obtain
\begin{equation}\label{step2}
\begin{split}
&(P\times m)\{x\in\Lambda |\, R_\omega=n\}=\frac12 E_\nu [2^{\alpha(\omega)} (x_n(\omega))^{\alpha(\omega)+1}]\\
&\le \frac12 E_\nu [2^{\alpha_0} (x_n(\omega))^{\alpha_0+1}]= 2^{\alpha_0-1} E_\nu [x_n(\omega)]^{\alpha_0+1}\\
&\le 2^{\alpha_0-1} \left(E_\nu [\mathbb 1_{\{n_1(\omega)\le n\}}\cdot x_n(\omega)^{\alpha_0+1}]+E_\nu [\mathbb 1_{\{n_1(\omega)> n\}}\cdot x_n(\omega)^{\alpha_0+1}]\right)\\
&=2^{2\alpha_0}c(\nu)^{-1-\frac{1}{\alpha_0}}
\frac{[\log n]^{\frac{q(\alpha_0+1)}{\alpha_0}}}{n^{\frac{1}{\alpha_0}+1}}+Ce^{-un^v}\le \hat C \frac{[\log n]^{\frac{q(\alpha_0+1)}{\alpha_0}}}{n^{\frac{1}{\alpha_0}+1}}.
\end{split}
\end{equation}
Choosing $\hat b=\frac{q(\alpha_0+1)}{\alpha_0}$ and $a=\frac{1}{\alpha_0}$ completes the verification of (P7).
\begin{theorem} \label{thm:lsv_main} Let $0< \alpha_0 < \alpha_1 < 1$ be fixed and
$\Omega = [\alpha_0, \alpha_1]^\mathbb Z$ equipped with product probability measure
$P:=\nu^\mathbb Z$ and left shift $\sigma$. Let $f_\omega,~  \omega \in \Omega$ be a random family of LSV maps with 
respect to the measure $\nu^{\mathbb Z}$.
Assume condition (A1) holds for the asymptotic expectations. 
 Then there exists a family of 
absolutely continuous sample stationary measures $\mu_\omega$ on $[0,1]$, for almost every 
$\omega \in \Omega$ (i.e.    ${f_{\omega}}_* \mu_\omega = \mu_{\sigma \omega}$). The system $\{ f_\omega, \mu_\omega\}_{\omega \in \Omega}$ is mixing, i.e., setting $\mu = \int_\Omega \mu_\omega dP$,
for all $\varphi, \psi \in L^2(\mu)$, 
$$
\lim_{n\to \infty}\left|\int_\Omega\int_0^1\varphi_{\sigma^n\omega}\circ f^n_\omega\cdot\psi_\omega d\mu_\omega dP -\int_\Omega\int_0^1\varphi_{\omega}d\mu_\omega dP\int_\Omega\int_0^1\psi_{\omega}d\mu_\omega dP\right| =0.
$$

Moreover, for every $\delta >0$ there exists a full measure subset $\Omega_0\subset \Omega$ such that for every $\omega \in \Omega_0$ there exists $C_\omega < \infty$ so that for any $\varphi\in L^\infty[0,1]$, $\psi \in C^\eta[0,1]$, the class of $\eta-$ H\"older functions on $[0,1]$, we have

\begin{itemize}
\item[(i)]{``Future" correlations} :  
$$
\left| \int (\varphi_{\sigma^n\omega} \circ f^{n}_\omega)\psi_\omega\, d\mu_\omega-\int \varphi _{\sigma^n\omega}d\mu_{\sigma^n\omega}\int \psi_{\omega} d\mu_{\omega}\right| \le C_\omega C_{\varphi,\psi }\frac{1}{n^{\frac{1}{\alpha_0} -\delta -1}}.
$$
\item[(ii)]{``Past" correlations} :  
$$
\left| \int (\varphi_\omega \circ f^{n}_{\sigma^{-n}\omega})\psi_{\sigma^{-n}\omega}\, d\mu_{\sigma^{-n}\omega}-\int \varphi_\omega \, d\mu_{\omega}\int \psi_{\sigma^{-n}\omega} d\mu_{\sigma^{-n}\omega}\right| \le C_\omega C_{\varphi,\psi}\frac{1}{n^{\frac{1}{\alpha_0} -\delta -1}}.
$$
\end{itemize}
Finally, there exist constants $C>0$, $u'>0$ and $0< v' <1$ such that the random variable $C_\omega$ satisifies the following 
tail estimates: for all $n \in \mathbb N$
$$
P\{ C_\omega >n\} \leq Ce^{-u'n^{v'}}.
$$
In particular, $C_\omega$ is integrable. Every $\eta\in(0,1]$ can be used by choosing\footnote{Recall that $\gamma$ is the regularity parameter in the distortion condition (P2).} $\gamma\in[\frac12, 1)$ so that $2^{\eta}\gamma\ge1$.
\end{theorem}

\begin{proof}
Proposition \ref{prop:applic_general} establishes conditions (P1) - (P3) and (P5). 
Since we are assuming (A1), condition (P4) follows from equations \eqref{eq:R_for_xn} and \eqref{eq:upperxnomega}
with constants $b=\frac{q}{\alpha_0}, ~a = \frac{1}{\alpha_0}, ~C=c(\nu)^{-\frac{1}{\alpha_0}}$. The second condition in (P4) holds because of equation \eqref{eq:n1}. \added[id=WB]{(P7) is verified in \eqref{step2}}.  

Theorem \ref{thm:acsm} therefore applies and gives existence and mixing of the sample stationary measures $\mu_\omega$.   Finally we apply Theorem \ref{thm:DC} to obtain the decay of correlations. 

 For $\vp\in L^\infty([0,1])$ and $\psi \in  C^\eta([0,1])$ define  $\bpsi =\psi\circ \pi_\omega, \bvp=\vp\circ \pi_\omega:\Dom \to \mathbb R$. Then we have  $\int (\vp\circ f^{n}_\omega)\psi d\mu_\omega= \int (\bvp\circ F^n_\omega)\bpsi h_\omega dm $.  Now, to apply Theorem \ref{thm:DC} it is sufficient to show  $\bpsi h_\omega\in \mF^{K_\omega}_\gamma$ and $\bvp\in \mL_\infty^{K_\omega}$. The latter is obvious, since 
the projection $\pi_\omega$ is 
nonsingular and hence, for $\nu_\omega-$  a.e.\ 
$(x, \ell)$ we have  $|\bvp(x, \ell)| \le \|\vp\|_{L^\infty}$. 
For the former one we first note that since $|(f^{R_\omega}_\omega)'|\geq2$ we have $|x-y| \le (\frac{1}{2})^{s(x, y)}$. Hence,  for any $(x, \ell), (y, \ell)\in \Dom$ we have the inequality 
\begin{equation}\label{eq:step1}
|\bpsi (x, \ell)- \bpsi (y, \ell)| \le\|\psi \|_\eta |x-y |^\eta \leq \|\psi\|_\eta ({\frac12})^{\eta\cdot s(x,y)}\le\|\psi\|_\eta(\frac{1}{2^\eta\gamma})^{s(x,y)}\gamma^{s(x,y)}.
\end{equation}
Now since $s((x, \ell), (y, \ell)) = s(x,y)$  the inequality \eqref{eq:step1} implies
\begin{align*}
&|(\bpsi h_\omega)(x, \ell)- (\bpsi h_\omega)(y, \ell)| \le {\|\psi\|}_{\infty} {\|h\|}_{\mF_\gamma^{K_\omega}}\gamma^{s((x, \ell), (y, \ell))}   
\\& + {\|h_\omega\|}_{\mL_\infty}\|\psi\|_\eta\gamma^{s((x, \ell), (y, \ell))}.
\end{align*}
 
\end{proof}

\subsection{Sharp asymptotics on the measure of return-time intervals}

Although we will not need lower bound estimates on the $x_n(\omega)$ to prove the main results in this 
paper, it is not difficult to identify conditions (see Assumption (A2) below) under which the upperbounds from the previous section are sharp.  This condition will hold for all of the examples discussed in this paper. 

Notice that from equation (\ref{eq:upperxnomega}) and the summability derived in 
equation (\ref{eq:n1}), an application of Borel-Cantelli yields, for almost every $\omega$,  
$$\liminf_\ell  \frac{(\log\ell)^q}{\ell[x_{\ell}(\omega)]^{\alpha_0}}  \geq c.$$
Keeping in mind that $0<c<c(\nu)$ was arbitrary (and working through a sequence of 
choices $c$ increasing to $c(\nu)$, applying Borel-Cantelli at each step) we obtain a set 
$\Omega_2 \subseteq \Omega$ of full $P-$measure such that for every 
$\omega \in \Omega_2 $ we have

\begin{equation}\label{eq:limsupxnomega}
\limsup_\ell \frac {\ell^{\frac{1}{\alpha_0}} x_\ell( \omega)}{[\log \ell]^{\frac{q}{\alpha_0}}}
\leq \frac{1}{[c(\nu)]^{\frac{1}{\alpha_0}}}.
\end{equation}

\deleted[id=CB]{Although the above expression appears to depend on the parameter $n$, we may remove this dependence by taking a suitable intersection over full 
measure subsets of $\Omega$.  That is, we have shown that equation (\ref{eq:limsupxnomega}) holds for every $n$ and almost every $\omega$. }

Now we concentrate on deriving lower bounds.   Using definition \eqref{def:lsv} and 
the estimate $(1 + x)^{-\alpha} \geq 1 - \alpha x$, valid for  
$\alpha>0,~ x \geq 0$, and by substituting $x = [2x_n(\omega)]^{\alpha(\omega)}$ we obtain
\begin{equation}\label{eq:onestepl}\begin{split}
\frac{1}{[x_{n}(\omega)]^{\alpha_0}} -\frac{1}{[x_{n-1}(\sigma \omega)]^{\alpha_0}} &\leq 
\alpha_0 2^{\alpha_0} [2x_n(\omega)]^{\alpha(\omega) - \alpha_0}.\\
\end{split}
\end{equation}
Iterating \eqref{eq:onestepl} we obtain
\begin{equation}\begin{split}\label{eq:xnsigma_n_lower}
\frac{(\log \ell)^q}{\ell [x_{\ell}(\omega)]^{\alpha_0}} \le \frac{(\log \ell)^q}{\ell}2^{\alpha_0} &+ \alpha _02^{\alpha_0}\frac{(\log \ell)^q}{\ell} 
\sum_{k=2}^\ell[2x_k(\sigma^{\ell-k}\omega)]^{\alpha(\sigma^{\ell-k}\omega) - \alpha_0}\\
= \frac{(\log \ell)^q}{\ell} 2^{\alpha_0} &+ \alpha _02^{\alpha_0} \frac{(\log \ell)^q}{\ell}
\sum_{k=2}^{\lfloor \sqrt \ell \rfloor}[2x_k(\sigma^{\ell-k}\omega)]^{\alpha(\sigma^{\ell-k}\omega) - \alpha_0}\\
+\alpha _02^{\alpha_0}\frac{(\log \ell)^q(\ell-\lfloor \sqrt \ell \rfloor)}{\ell \log(\ell - \lfloor \sqrt \ell \rfloor)^q}&\cdot\frac{\log(\ell - \lfloor \sqrt \ell \rfloor)^q}{\ell - \lfloor \sqrt \ell \rfloor}\sum_{k=\lfloor \sqrt \ell \rfloor +1}^ \ell[2x_k(\sigma^{\ell-k}\omega)]^{\alpha(\sigma^{\ell-k}\omega) - \alpha_0}\\
=(I) &+ (II) + (III).\\
\end{split}
\end{equation}
Clearly $(I) = o(1)$ and the same is true of $(II)$ since  
$$(II) \leq \alpha _02^{\alpha_0} \frac{(\log \ell)^q}{\lfloor \sqrt \ell \rfloor}.$$
In order to estimate $(III)$ note that from equation \eqref{eq:limsupxnomega}, for all $\ell$ sufficiently large (depending on $\omega$), for $ \lfloor \sqrt \ell \rfloor + 1 \leq k \leq \ell$, 
$$ \alpha _0 2^{\alpha_0}[2x_k(\sigma^{\ell-k}\omega)]^{\alpha(\sigma^{\ell-k}\omega) - \alpha_0} \leq \alpha _0 2^{\alpha_0}\left[3\frac{(\log k)^{\frac{q}{\alpha_0}}}{[c(\nu)]^{\frac{1}{\alpha_0}} k^{\frac{1}{\alpha_0}}}\right]^{\alpha(\sigma^{\ell-k}\omega) - \alpha_0} =: A^\prime_k(\omega).$$

From now on we write $A'_k:=A'_k(\omega)$. In addition to Assumption (A1) we now assume\footnote{We will see that for many examples, including the ones presented in the next section, Assumption (A2) will hold.} the following asymptotics on the $E_\nu(A_k^\prime)$:\\

\bigskip

\noindent {\bf Assumption (A2) (Asymptotics on expectations revisited)}
\begin{equation}\label{eq:Akprime}
\frac{\log(\ell - \lfloor \sqrt \ell \rfloor)^q}{\ell - \lfloor \sqrt \ell \rfloor}\sum_{k=\lfloor \sqrt \ell \rfloor +1}^ \ell E_\nu(A_k^\prime) \rightarrow c(\nu).
\end{equation} 
\smallskip


Another large deviations estimate as in the preceding section will give, for each $c(\nu) < c$ an
integer $N_2 = N_2(c)$ so that for all $\ell > N_2$,  
$\frac{\log(\ell - \lfloor \sqrt \ell \rfloor)^q}{\ell - \lfloor \sqrt \ell \rfloor}\sum_{k=\lfloor \sqrt \ell \rfloor +1}^ \ell E_\nu(A_k^\prime) \leq \frac{c+c(\nu)}{2}$ and 
$$P\left\{\frac{\log(\ell - \lfloor \sqrt \ell \rfloor)^q}{\ell - \lfloor \sqrt \ell \rfloor} \sum_{k= \lfloor \sqrt \ell \rfloor +1}^{ \ell }A_k^\prime > c \right\}
\leq \exp\left[-\frac{(\ell - \lfloor \sqrt \ell \rfloor) [c-c(\nu) ]^2}{8r_0 (\log ( \ell - \lfloor \sqrt \ell \rfloor)^{2q}}\right].
$$
Once again, application of Borel-Cantelli implies there exists a random variable $n_2(\omega)$, finite almost everywhere,  such that 
for all $\ell > n_2(\omega)$
\begin{equation}\label{eq:xnloweromega} 
\frac{(\log \ell)^q}{\ell [x_{\ell}(\omega)]^{\alpha_0}}  \leq c^\prime
\end{equation}
and for each $0<v<1$ 
there are constants $u>0$ and $C < \infty$ so that 
$$P\left\{ n_2( \omega) > n  \right\} \leq Ce^{-un^v}. $$
The factor $c^\prime> c$ in equation (\ref{eq:xnloweromega}) is necessary to account for the two  $o(1)$ terms $(I)$ and $(II)$ in equation (\ref{eq:xnsigma_n_lower}). 
Returning to equation (\ref{eq:xnloweromega}), another sequence of Borel-Cantelli reductions
over the parameter $c^\prime$ decreasing to $c(\nu)$ gives a full measure set $\Omega_2 \subseteq \Omega$ such that 
for every $\omega \in \Omega_2$
\begin{equation}\label{eq:liminfxnomega}
\liminf_\ell \frac {\ell^{\frac{1}{\alpha_0}} x_\ell( \omega)}{[\log \ell]^{\frac{q}{\alpha_0}}}
\geq \frac{1}{[c(\nu)]^{\frac{1}{\alpha_0}}}.
\end{equation}
\deleted[id=CB]{The dependence of this expression on $n$ can be removed by taking a further intersection over 
full measure sets as was done for the upper bounds at the beginning of this section. } We have therefore established the following
(fibre-wise, or quenched) exact asymptotics

\begin{proposition} \label{prop:exact_asympt}
For random LSV maps as described in Theorem \ref{thm:lsv_main}, assuming asymptotic 
growth conditions  (\ref{eq:Ak_asympt}) and  (\ref{eq:Akprime}), we have the following exact asymptotics:
There is a full measure subset $\Omega^\prime \subseteq \Omega$ such that for all $\omega \in \Omega^\prime$, for all $n=0, 1, 2, \dots$ we have
$$x_\ell(\omega)\sim \left[\frac{(\log\ell)^q}{c(\nu)\ell}\right]^{\frac{1}{\alpha_0}}.$$\end{proposition}

\subsection{Natural probability distributions on the parameter space}
In this subsection we verify assumptions (A1) and (A2) for some natural probability distributions $\nu$ on 
$[\alpha_0, \alpha_1]$.
\subsubsection{Example:  Discrete distribution. } 
Here we assume $\nu$ is a discrete probability distribution; for concreteness, 
$\nu = p_1 \delta_{\alpha_0} + p_2 \delta_{\alpha_1}$ with $p_i >0$ and  
$p_1 + p_2 = 1$.  

\begin{lemma} 
$\frac{1}{\ell}\sum_{k=1}^\ell E_\nu A_k \rightarrow \alpha_02^{\alpha_0}p_1.  $
Therefore, in condition (\ref{eq:Ak_asympt}) we can take $q:=0$ and 
$c(\nu) :=\alpha_02^{\alpha_0} p_1 > 0$.
\end{lemma}
\begin{proof} 
Write  $A_k(\omega)= X_k(\omega) - Y_k(\omega)$ where
\begin{equation}\label{eq:Ak=Xk-Yk}\begin{split}
X_k(\omega) &=\alpha_02^{\alpha_0} \left[\frac{2c_k(\alpha_0)}{k^{\frac{1}{\alpha_0}}}\right]^{\alpha(\sigma^{n-k}\omega) - \alpha_0},\\
Y_k(\omega) &= \frac{1+ \alpha_0}{2} \alpha_02^{\alpha_0}\left[\frac{2c_k(\alpha_1)}{k^{\frac{1}{\alpha_1}}}\right]^{2\alpha(\sigma^{n-k}\omega) - \alpha_0}.\\
\end{split}
\end{equation}
Using $\sigma$-invariance of $P$, a direct calculation shows 
$$
E_\nu (X_k) = \alpha_02^{\alpha_0}\left\{ p_1 + p_2\left[\frac{2c_k(\alpha_0)}{k^{\frac{1}{\alpha_0}}}\right]^{\alpha_1-\alpha_0}\right\},
$$
while 
$$
E_\nu Y_k = \frac{1+ \alpha_0}{2} \alpha_02^{\alpha_0}\left\{ p_1\left[\frac{2c_k(\alpha_1)}{k^{\frac{1}{\alpha_1}}}\right]^{ \alpha_0} + p_2\left[\frac{2c_k(\alpha_1)}{k^{\frac{1}{\alpha_1}}}\right]^{2\alpha_1 - \alpha_0}
\right\}.
$$
Therefore, 
$$\frac{1}{\ell-1}\sum_{k=2}^\ell  E_\nu X_k = \alpha_02^{\alpha_0} p_1 + O(\ell^{ 1- \frac{\alpha_1}{\alpha_0}}),$$
whereas
$$\frac{1}{\ell-1}\sum_{k=2}^\ell  E_\nu Y_k =  O(\ell^{ -\delta })$$
for $\delta=\min\{\alpha_0/\alpha_1, 2-\alpha_0/\alpha_1\}>0$. 
It follows that 
$$ \frac{1}{\ell}\sum_{k=1}^\ell E_\nu A_k =  \frac{\ell -1}{\ell} \alpha_02^{\alpha_0} p_1 + O(\ell^{ -\zeta}),$$
where $\zeta:= \min\{ \delta, 1- \frac{\alpha_1}{\alpha_0}\} > 0$.  Therefore 
Assumption (\ref{eq:Ak_asympt}) is satisfied by taking $q:=0$ and 
$c(\nu) := \alpha_0 2^{\alpha_0}p_1 >0$.
\end{proof}

We now show the upperbound on the $x_\ell( \omega)$ obtained above is sharp. 
\begin{proposition}[Sharp asymptotics for the discrete probability distribution] For almost every $\omega$
$$x_\ell(\omega)\sim \left[\frac{1}{c(\nu)\ell}\right]^{\frac{1}{\alpha_0}},$$
where $c(\nu)=\alpha _02^{\alpha_0}p_1$. 
\end{proposition}
\begin{proof}
We only need to verify Assumption (A2) and apply Proposition \ref{prop:exact_asympt}. 
\begin{equation*}\begin{split}
\frac{\log(\ell - \lfloor \sqrt \ell \rfloor)^q}{\ell - \lfloor \sqrt \ell \rfloor}\sum_{k=\lfloor \sqrt \ell \rfloor +1}^ \ell E_\nu(A_k^\prime)
&= \frac{\alpha _02^{\alpha_0}}{\ell - \lfloor \sqrt \ell \rfloor}\sum_{k=\lfloor \sqrt \ell \rfloor +1}^ \ell p_1 + p_2\left[ \frac{1}{c(\nu)k}
\right]^{\frac{\alpha_1}{\alpha_0} - 1} \\
&= \alpha _02^{\alpha_0}p_1 + O\left(\ell^{\frac{\alpha_0-\alpha_1}{2\alpha_0}}\right).\\
\end{split}
\end{equation*}
Since $\alpha _02^{\alpha_0}p_1=c(\nu)$ we have verified Assumption (A2). 
\end{proof}

\subsubsection{Example: Uniform distribution. } 
Here we assume $\nu$ is the uniform probability distribution on 
$[\alpha_0, \alpha_1]$, that is, for $\alpha_0 < t < \alpha_1$, 
$$\nu([\alpha_0, t]) = \frac{1}{\alpha_1 - \alpha_0}(t - \alpha_0).$$
We start with a lemma that will allow us to compute the appropriate expectations 
in condition (\ref{eq:Ak_asympt}). 

\begin{lemma}\label{normalization} 
Let $c\ge 1$. Then, as $u \rightarrow \infty$
\begin{equation}
E_\nu \left[e^{-(c\alpha(\omega)-\alpha_0)u}\right]\sim
\frac{1}{\alpha_1-\alpha_0}\cdot\frac{1}{cu}\cdot e^{-(c-1)\alpha_0u}. 
\end{equation}
\end{lemma}
\begin{proof}
We have 
\begin{align*}
E_\nu \left[e^{-(c\alpha(\omega)-\alpha_0)u}\right]&=
\frac{1}{\alpha_1-\alpha_0}\int_{\alpha_0}^{\alpha_1}e^{-(cx-\alpha_0)u}dx\\
&=\frac{1}{\alpha_1-\alpha_0}\cdot\frac{1}{cu}\left[e^{-(c-1)\alpha_0u}-e^{-(c\alpha_1-\alpha_0)u}\right]\\
&= \frac{1}{\alpha_1-\alpha_0}\cdot\frac{1}{cu}e^{-(c-1)\alpha_0u}\left[ 
1-e^{-c(\alpha_1 - \alpha_0)u}\right ].\\
\end{align*}
\end{proof}
As in the previous section, we decompose $A_k(\omega) := X_k(\omega) -
Y_k(\omega)$ according to equation (\ref{eq:Ak=Xk-Yk}). 

Using $\sigma$ invariance of $P$  and Lemma \ref{normalization}
with $u= -\log\left(\frac{2c_k(\alpha_0)}{k^{\frac{1}{\alpha_0}}}\right), ~ c=1$ and 
$u= -\log\left(\frac{2c_k(\alpha_1)}{ k^{\frac{1}{\alpha_1}}}\right),~ c=2$, respectively, we obtain
\begin{equation}\begin{split}
E_\nu (X_k(\omega) ) &\sim \frac{\alpha_0^2 2^{\alpha_0}}{\alpha_1-\alpha_0} \frac{1}{\log k},\\
E_\nu (Y_k(\omega) )&= o(\log k).\\
\end{split}
\end{equation}
It follows that  $\log k \cdot E_\nu A_k \sim \frac{\alpha_0^2 2^{\alpha_0}}{\alpha_1 - \alpha_0}$. 
Now apply Lemma \ref{appendix2} of the Appendix to compute the asymptotics for the sum:
$$ \frac {\log \ell}{\ell} \sum_{k=1}^\ell  E_\nu A_k \rightarrow \frac{\alpha_0^2 2^{\alpha_0}}{\alpha_1 - \alpha_0}.$$
Therefore we can take $q=1$ and $c(\nu) = \frac{\alpha_0^2 2^{\alpha_0}}{\alpha_1 - \alpha_0}$ 
in Assumption (A1)

\begin{proposition}[Sharp asymptotics for the uniform probability distribution]
For almost every $\omega$
$$x_\ell(\omega)\sim \left[\frac{\log \ell}{c(\nu)\ell}\right]^{\frac{1}{\alpha_0}},$$
where $c(\nu) = \frac{\alpha_0^2 2^{\alpha_0}}{\alpha_1 - \alpha_0}$.
\end{proposition}
\begin{proof}
We verify Assumption (A2) and apply Proposition \ref{prop:exact_asympt}. 
\begin{equation*}\begin{split}
\frac{\log(\ell - \lfloor \sqrt \ell \rfloor)}{\ell - \lfloor \sqrt \ell \rfloor}\sum_{k=\lfloor \sqrt \ell \rfloor +1}^ \ell E_\nu(A_k^\prime)
&= \alpha _02^{\alpha_0}\frac{\log(\ell - \lfloor \sqrt \ell \rfloor)}{\ell - \lfloor \sqrt \ell \rfloor}\\
\times
\sum_{k=\lfloor \sqrt \ell \rfloor +1}^\ell 
&E_\nu \left(\left[ 3\frac{\log k}{c(\nu)k}
\right]^{\frac{\alpha(\sigma^{n-k}\omega) - \alpha_0}{\alpha_0}} \right). 
\end{split}
\end{equation*}
We can evaluate the individual expectations using Lemma \ref{normalization} to obtain:
$$E_\nu \left(\left[ 3\frac{\log k}{c(\nu)k}
\right]^{\frac{\alpha(\sigma^{n-k}\omega) - \alpha_0}{\alpha_0}} \right)\sim 
\frac{\alpha_0}{\alpha_1 - \alpha_0} \frac{1}{\log k}.$$
Now, two applications of Lemma \ref{appendix2} from the Appendix shows
$$\sum_{k=\lfloor \sqrt \ell \rfloor +1}^ \ell  \frac{1}{\log k}=
\sum_{k=1}^ \ell  \frac{1}{\log k} -\sum_{k=1}^{ \lfloor \sqrt \ell \rfloor } \frac{1}{\log k} 
\sim \frac{\ell - 2 \lfloor \sqrt \ell \rfloor}{\log \ell}.  
$$
Applying this to the first estimate gives
$$
\frac{\log(\ell - \lfloor \sqrt \ell \rfloor)}{\ell - \lfloor \sqrt \ell \rfloor}\sum_{k=\lfloor \sqrt \ell \rfloor +1}^ \ell E_\nu(A_k^\prime)
\sim \frac{\alpha_0^2 2^{\alpha_0}}{\alpha_1 - \alpha_0}  \frac{\log(\ell - \lfloor \sqrt \ell \rfloor)}{\ell - \lfloor \sqrt \ell \rfloor}
\frac{\ell - 2 \lfloor \sqrt \ell \rfloor}{\log \ell} \rightarrow \frac{\alpha_0^2 2^{\alpha_0}}{\alpha_1 - \alpha_0}. 
$$
Since $\frac{\alpha_0^2 2^{\alpha_0}}{\alpha_1 - \alpha_0}= c(\nu)$ we have verified Assumption (A2).
\end{proof}


\subsubsection{Example: Quadratic distribution. } 
Here we assume $\nu$ is the quadratic probability distribution on 
$[\alpha_0, \alpha_1]$, given by  
$$\nu([\alpha_0, t]) = \frac{1}{(\alpha_1 - \alpha_0)^2}(t - \alpha_0)^2.$$
Again, we begin with a simple lemma that will allow us to estimate the 
expectations. 

\begin{lemma}\label{normalization1} 
Let $c\ge 1$. Then, as $u \rightarrow \infty$
\begin{equation}
E_\nu \left[e^{-(c\alpha(\omega)-\alpha_0)u}\right]\sim
\frac{2}{(\alpha_1-\alpha_0)^2}\cdot\frac{1}{(cu)^2}\cdot e^{-(c-1)\alpha_0u}. 
\end{equation}
\end{lemma}
\begin{proof}
We have 
\begin{align*}
&E_\nu \left[e^{-(c\alpha(\omega)-\alpha_0)u}\right]=
\frac{2}{(\alpha_1-\alpha_0)^2}\int_{\alpha_0}^{\alpha_1}e^{-(cx-\alpha_0)u}(x-\alpha_0)dx\\
&=\frac{2}{(\alpha_1-\alpha_0)^2}\left\{\frac{-1}{cu}e^{-(c\alpha_1 - \alpha_0)u}(\alpha_1-\alpha_0)
-\frac{1}{(cu)^2} [ e^{-(c\alpha_1-\alpha_0)u} - e^{-(c-1)\alpha_0 u} ]\right\}\\
&= \frac{2}{(\alpha_1-\alpha_0)^2}\cdot\frac{1}{(cu)^2}e^{-(c-1)\alpha_0u}\left\{ 
1-e^{-c(\alpha_1 - \alpha_0)u}[cu(\alpha_1-\alpha_0) +1] \right\}.
\end{align*}
\end{proof}
Writing $A_k = X_k - Y_k$ as in equation (\ref{eq:Ak=Xk-Yk}), using invariance of $P$ and Lemma \ref{normalization1} with we obtain
\begin{equation}\begin{split}
E_\nu (X_k(\omega) ) &\sim 2\left(\frac{\alpha_0}{\alpha_1-\alpha_0}\right)^2 \frac{1}{(\log k)^2},\\
E_\nu (Y_k(\omega) )&\sim 2\left(\frac{\alpha_1}{\alpha_1-\alpha_0}\right)^2\frac{k^{-\frac{\alpha_0}{\alpha_1}}}{ (\log k)^2} = o\left(\frac{1}{(\log k)^2}\right).\\
\end{split}
\end{equation}
It follows that $(\log k)^2 \cdot E_\nu A_k \sim 2\left(\frac{\alpha_0}{\alpha_1 - \alpha_0}\right)^2$
after which an application of Lemma \ref{appendix2} of the Appendix implies
$$
\frac{ (\log \ell)^2}{\ell}\sum_{k=1}^\ell E_\nu A_k \rightarrow 2\left (\frac{\alpha_0}{\alpha_1 - \alpha_0}\right)^2.
$$
We take $q=2$ and $c(\nu) = 2\left(\frac{\alpha_0}{\alpha_1 - \alpha_0}\right)^2$ in 
the assumption (\ref{eq:Ak_asympt}).
For this example, lower bounds can also be obtained by essentially following the steps in the previous 
example and using Lemma \ref{normalization1} in place of Lemma \ref{normalization}.
\begin{proposition}[Sharp asymptotics for the quadratic probability distribution]
For almost every $\omega$
$$x_\ell(\omega)\sim \left[\frac{(\log \ell)^2}{c(\nu)\ell}\right]^{\frac{1}{\alpha_0}},$$
where $c(\nu) = 2\left(\frac{\alpha_0}{\alpha_1 - \alpha_0}\right)^2$.
\end{proposition}
\begin{proof}
We verify Assumption (A2) and apply Proposition \ref{prop:exact_asympt}.  The key estimates are 
$$E_\nu A_k^\prime \sim \frac{2 \alpha_0^2}{(\alpha_1-\alpha_0)^2} \frac{1}{(\log k)^2}$$
and
$$
\frac{\log(\ell - \lfloor \sqrt \ell \rfloor)^2}{\ell - \lfloor \sqrt \ell \rfloor}\sum_{k=\lfloor \sqrt \ell \rfloor +1}^ \ell E_\nu(A_k^\prime)
\sim \frac{2 \alpha_0^2}{(\alpha_1-\alpha_0)^2}\frac{(\log(\ell - \lfloor \sqrt \ell \rfloor))^2}{\ell - \lfloor \sqrt \ell \rfloor}
\frac{(\ell - 4\lfloor \sqrt \ell \rfloor)}{(\log \ell)^2} \sim c(\nu),
$$
where we have again used Lemma \ref{appendix2} in the Appendix to estimate the sum and verify 
Assumption (A2) for this example. 
\end{proof}
\section{Proof of existence of absolutely continuous mixing sample measures}\label{sec:existence}
\subsection{Measure of the tail of the tower} 

\begin{lemma}
 $\mL^{K_\omega}_\infty \subset L^2(\Delta, m):=\{\vp:\Delta\to \mathbb R \mid \|\vp\|_{L^2}:= \int |\vp|^2dm_\omega dP< \infty\}$.
\end{lemma}

\begin{proof} 
The  bounds on $P\{K_\omega> n\}$ imply that $K_\omega \in L^2(\Omega, P)$. Hence 
$$
\|\vp\|_{L^2} \le C_\vp \int K_\omega^2 dP<\infty.
$$
\end{proof} 

\subsection{Distortion estimates} 
Here we prove consequences of bounded distortion which are key for many of the later computations. 
For any $n\ge 1$
\begin{equation}\label{pn}
\mathcal P_\omega^{(n)}=\vee_{i=0}^{n-1}F_\omega^{-i}\mP_{\sigma^i\omega}   \quad\text{and} \quad 
\mathcal A_\omega^{(n)}=\{A\in \mathcal P_{\omega}^{(n)} \mid F^n_\omega A=\Delta_{\sigma^{n}\omega, 0}\}.
\end{equation}
\begin{lemma}\label{bddn}
For any $n\ge 1$,  $A\in \mathcal A_\omega^{(n)}$ and $x, y \in A$ the following inequality holds

  \begin{equation*}
 \left|\frac{JF^n_\omega(x)}{JF^n_\omega(y)}-1\right|\leq D,
    \end{equation*}
 where $D$ is  as in   \eqref{bddd}.
\end{lemma}

 \begin{proof}
 The collection $\mathcal A_\omega^{(n)}$ is a partition of $F^{-n}_\omega\Delta_{\sigma^n\omega, 0}$ and for any $x\in\Delta_{\sigma^n\omega, 0}$
each $A\in\mathcal A_\omega^{(n)}$ contains a single element  of $\{F^{-n}_\omega x\}.$
 For $x\in A$ let $j(x)$ be the number of visits of its orbit to $\Delta_{\sigma^k\omega, 0}$ up to time \(  n  \). Since the images of $A$
 before time $n$ will remain in an element of $\mathcal P_{\sigma^k\omega},$ all the points in $A$ have the same itineraries,   up to time \(  n  \)  and so  $j(x)$ is constant on $A.$
Therefore $JF^n_\omega(x)=(JF^{R_\omega}_\omega)^j(\tilde{x}),$ for the projection $\tilde{x}$ of $x$ onto
 $\Delta_{\sigma^n\omega, 0}$ (i.e. if $x=(z, \ell)$ then $\tilde{x}=(z,0)$).
Thus for any $x,y\in A$ from \eqref{bddd} we obtain
  \begin{equation}
 \left|\frac{JF^n_\omega(x)}{JF^n_\omega(y)}-1\right|=
 \left|\frac{(JF^{R_\omega}_\omega)^j(\tilde{x})}{(JF^{R_\omega}_\omega)^j(\tilde{y})}-1\right|\leq  D.
    \end{equation}
 \end{proof}
\begin{corollary}\label{connec}
For any $A\in\mathcal A_\omega^{(n)}$, $F_\omega^n: A \to \Delta_{{\sigma^n \omega}, 0}$ is a 
bijection and for each $y\in A$ we have
 \begin{equation}
JF^n_\omega(y)\geq \frac{m(\Delta_{\sigma^n\omega, 0})}{m(A)(1+D)}.
\end{equation}
\end{corollary}

\begin{proof}
Lemma \ref{bddn} implies $JF^n_\omega(x)\leq (1+D)JF^n_\omega(y).$ Integrating both sides of this inequality with
respect to $x$ over  $A$  gives
$$
m(\Delta_{\sigma^n\omega, 0})=\int_AJF^n_\omega(x)dm\leq JF^n_\omega(y)(1+D)m(A),
$$
which finishes the proof.
\end{proof}

\begin{lemma}\label{cbddd} \text{}
\begin{itemize}
\item[(i)] There exists a constant $M_0\ge 1$ such that for all $n\in \mathbb N$ and $\omega \in \Omega$, 
$$
\frac{d(F^n_\omega)_\ast m}{dm} \le M_0 m(\Delta_\omega) \le M_0M.
$$
\item[(ii)] Let $\lambda_\omega$ be a family of absolutely continuous probability measures on $\{\Delta_\omega\}$ with $\frac{d\lambda_\omega}{dm}\in\mathcal F^+_\gamma$.  For every  $A\in \mathcal A_\omega^{(n)}$ let $\nu_{\sigma^n\omega}=(F^n_\omega)_\ast(\lambda_\omega|A)$  and $\varphi_\omega=\frac{d\nu_{\sigma^n\omega}}{dm}$. There exists $C_\lambda>0$ such that for each $\omega\in\Omega$,   for all $x, y\in \Delta_{\sigma^n\omega, 0}$ we have 
 \begin{equation*}
  \left|\frac{\varphi_\omega(x)}{\varphi_\omega(y)}-1\right|\leq D+ [e^{C_{\lambda}}-1](1+D),
 \end{equation*}
 where $D$ is as in \eqref{bddd}.  \end{itemize} 
\end{lemma}

\begin{proof}To prove the  item (i) we estimate 
   the density ${d(F^n_\omega)_\ast m}/{dm}$ at an arbitrary point $x\in \Delta_{\sigma^n\omega}$  and consider three different cases according to the position of $ x$. First of all, for any $x\in\Delta_{\sigma^n\omega, 0}$,  from Corollary \ref{connec} we have

   \begin{align*}
  \frac{d(F^n_\omega)_\ast m}{dm}(x)=\sum_{y\in F^{-n}_\omega x}\frac{1}{JF^n_\omega(y)}&\leq
  (D+1)\sum_{A\in \mathcal A_n}\frac{m(A)}{m(\Delta_{\sigma^n\omega, 0})}\\ &\leq
  (D+1)\frac{m(\Dom)}{m(\Delta_{\sigma^n\omega, 0})}.
 \end{align*}
 Since $m(\Delta_{\sigma^n\omega, 0}) = m(\Lambda)=1$ choosing $M_0= {D+1}$ finishes the proof 
for the case $x\in\Delta_{\sigma^n\omega, 0}.$

 For $x\in\Delta_{\sigma^{n}\omega, \ell}$ with $\ell\ge n$ we have $F^{-n}_\omega(x)=y\in\Delta_{\omega, \ell-n}.$
 Since $JF_\omega(y)=1$ for any $y\in\Dom\setminus\Domo,$
  $$\frac{d(F^n_\omega)_\ast m}{dm}(x)=\frac{1}{JF_\omega(y)\cdots JF_{\sigma^{n-1}\omega}(F^{n-1}_\omega y)}=1.$$
Finally,
let $x\in\Delta_{\sigma^n\omega, \ell},$ for  $0<\ell<n.$ Then for any $y\in F_\omega^{-n}x$ the
equality $F^{n-\ell}_\omega y=F^{-\ell}_{\sigma^{n-\ell}\omega} x\in\Delta_{\sigma^{n-\ell}\omega, 0}$ holds. Hence,
$JF_{\sigma^{n-\ell+j}\omega}(F^{j}_\omega y)=1$ for all $j=0, \dots, \ell-1.$ Therefore, by the chain
rule we obtain $JF^n_\omega(y)=JF^{n-\ell}_\omega(y).$ Hence the
problem is reduced to the first case since 
 $$\frac{d(F^n_\omega)_\ast m}{dm}(x)=\sum_{y\in F^{-n}_\omega x}\frac{1}{JF^{n}_\omega y}=
 \sum_{y\in F^{\ell-n}_\omega(F^{-\ell}_{\sigma^{n-\ell}\omega}x)}\frac{1}{JF^{n-\ell}_\omega y}=
 \frac{d(F^{n-\ell}_\omega)_\ast m}{dm}(F^{-\ell}_{\sigma^{n-\ell}\omega}(x)).$$
This finishes the proof of item (i). 

To prove the item (ii) we first note that $F^n_\omega:A\to\Delta_{\sigma^n\omega, 0}$ is invertible.
So for any $x\in\Delta_{\sigma^n\omega, 0}$ there is a unique $x_0\in A$ such that $F^n_\omega(x_0)=x$ and
$$
\frac{d\nu_{\sigma^n\omega}}{dm}(x)=\frac{1}{JF^n_\omega (x_0)}\frac{d\lambda_\omega}{dm}(x_0).
$$ 
Let $\varphi_\omega=\frac{d\nu_{\sigma^n\omega}}{dm}$
then for $x,y\in\Delta_{\sigma^n\omega, 0}$, using  Lemma \ref{bddn} and assumption on $\frac{d\lambda_\omega}{dm}$ we obtain
$$\left|\frac{\varphi_\omega(x)}{\varphi_\omega(y)}-1\right|=\left|\frac{JF^n_\omega(y_0)}{JF^n_\omega(x_0)}\frac{\frac{d\lambda_\omega}{dm}(x_0)}{\frac{d\lambda_\omega}{dm}(y_0)}-1\right|\leq D+ [e^{C_{\lambda}\gamma^{s(x_0, y_0)}}-1](1+D).$$
\end{proof}
\begin{remark}\label{rem:indep_of_lambda}
It is important to note that the constant $C_\lambda$ does not depend on $\omega$.
Moreover,  if  $A$ and $n$ are such that 
the orbits  $F^j(x_0), F^j(y_0), j= 1, 2, \dots n$ see sufficiently many returns to the base, so that $C_\lambda\gamma^{s(x_0, y_0)}\leq \log 2$, then the upperbound in (ii) becomes $2D +1$. The elements 
$A$ for which this holds are independent of the starting measure $\lambda$. 

\end{remark}

\subsection{Proof of Theorem \ref{thm:acsm}} 
\begin{proof}[Proof of Existence]
Recall that $m(\Lambda)=1$. 
Recall the definitions of $\mP_\omega^{(j)}$ and $\mathcal A_\omega^{(j)}$ in \eqref{pn}, and for $A\in \mP_{\sigma^{-j}\omega}^{(j)}|\Delta_{\sigma^{-j}\omega, 0}$ let 
$$
\phi_{j, A}^{\omega} =\frac{d}{dm}\left(F^j_{\sigma^{-j}\omega})_\ast(m_{A})\right),
$$
where $m_{A}(B)=m(A\cap B)$. 
Clearly, $\phi_{j, A}^{\omega}$ is a density on $\Dom,$ such that $\phi_{j, A}^{\omega}|\Delta_{\omega, \ell}\equiv0$ for $\ell>j$.  
Below we consider two cases depending on $A$. First, notice that if  $A\in\mathcal A_{\sigma^{-j}\omega}^{(j)}$ then $F^j_{\sigma^{-j}\omega}:A\to \Domo$ is a bijection.
For $x, y \in \Domo$, let $x', y' \in A$ be such that $F^j_{\sigma^{-j}\omega}(x')=x$, and   $F^j_{\sigma^{-j}\omega}(y')=y$. By the choice of $A$ there exists $i$ such that $F^j_{\sigma^{-j}\omega}= (F^{R_\omega}_{\sigma^{-j}\omega})^i$. The bounded distortion condition implies that 

\begin{equation}\begin{split}\label{eq:density_reg}
\left|\log \frac{\phi_{j, A}^\omega(y)}{\phi_{j, A}^\omega(x)}\right|=
\left|\log \frac{JF^j_{\sigma^{-j}\omega}(x')}{JF^j_{\sigma^{-j}\omega}(y')}  \right|
\le \sum_{\ell=0}^{i-1} D\gamma^{s(x, y)+(i-\ell)-1} \le \frac{D}{1-\gamma} \gamma^{s(x, y)}.
\end{split}
\end{equation}
Notice that  the constant in equation \eqref{eq:density_reg}  is independent of $j$, $A$ and $\omega.$ 
Hence for all $x, y\in  \Dom$  letting $D'=e^{\frac{D}{1-\gamma}}$ we have 
$$
 \phi^\omega_{j, A}(y) \le D'\phi_{j, A}^\omega(x).
$$
Integrating both sides  of the latter inequality over $\Domo$ with respect to $x$  implies
\begin{equation}\label{phiomegajA}
\phi^\omega_{j, A}(y) \le D' \frac{m(A)}{m(\Lambda)}= D'm(A).
\end{equation}
On the other hand, if  $A\in \mP_{\sigma^{-j}\omega}^{(j)}|\Delta_{\sigma^{-j}\omega, 0}$ such that $F^j_{\sigma^{-j}\omega}(A)\subset \Delta_{\omega, \ell}$ for $\ell>0$ then  $\phi_{j, A}^{\omega} (x)=\phi_{j-\ell, A}^{\sigma^{-\ell}\omega} (F_{\sigma^{-\ell}\omega}^{-\ell}(x))$. Hence we can apply \eqref{phiomegajA}. Futher define
\begin{equation*} \begin{aligned}
\phi_{n}^\omega= \frac{d}{dm}\left( \frac{1}{n} \sum_{j=0}^{n-1}  (F^{j}_{\sigma^{-j}\omega})_\ast (m|\Delta_{\sigma^{-j}\omega, 0})\right).
\end{aligned}\end{equation*}
As above  $\phi_{n}^{\omega}|\Delta_{\omega, \ell}\equiv0$ for $\ell>j$.  For $x\in \Delta_{\omega, \ell}$, $\ell\le j$ we write  $\phi_{n}$ as a convex combination of $ \phi^\omega_{j, A}$ and obtain   
\begin{equation}
\label{eq:phi_n}
\phi_{n}^\omega(x) \le D'.
\end{equation}
Hence, $\{\phi^\omega_n\}_{n\in\mathbb N}$ is a uniformly bounded family.  Notice that if $\phi_{n}^\omega(x)=0$ and $\phi_{n}^\omega(y)\neq 0$ then $s(x, y)=0$. Taking this into account for all $y\in\Dom$ such that  $\phi_{n}^\omega(y)\neq 0$  we have
\begin{equation}
 \begin{aligned}
&|\phi_{n}^\omega(x)-\phi_{n}^\omega(y)|\le \frac{1}{n}\sum_{j=0}^{n-1}\sum_{A \in \mP_{\sigma^{-j}\omega}^{(j)}|\Delta_{\sigma^{-j}\omega, 0}} |\phi_{j, A}^\omega(y)|\left| \frac{\phi_{j, A}^\omega(x)}{\phi_{j, A}^\omega(y)}-1 \right| 
\\& \le D'\frac{(1-\gamma)(D' -1)}{D} \frac{1}{n}\sum_{j=0}^{n-1}\sum_{A \in \mP_{\sigma^{-j}\omega}^{(j)}|\Delta_{\sigma^{-j}\omega, 0}}\left| \log \frac{\phi_{j, A}^\omega(y)}{\phi_{j, A}^\omega(x)}\right| m(A) \le D'(D'-1) \gamma^{s(x, y)},
\end{aligned}
\end{equation}
where we have used equation \eqref{eq:density_reg} in the last step. 

Hence $\phi_n \in \mF^+_\gamma\cap \mF^1_\gamma$ (i.e. $K_\omega\equiv1$). Since, $d(x, y):= \gamma^{s(x, y)}$ defines separable metric space structure on $\Dom$ for each $\omega$, we can find a subsequence $\phi_{n_k} \in \mF^+_\gamma\cap \mF^1_\gamma$ which is convergent pointwise.  By diagonal argument we then construct convergent subsequence $\{\phi^{\sigma^\ell\omega}_{n_k} dx\}$ for every $\ell \in \mathbb Z$.  The limiting measure is  $F_\omega$-equivariant  i.e.  $(F_\omega)_*\nu_\omega=\nu_{\sigma\omega}$. Moreover,    $h_\omega :=\frac{d\nu_\omega}{dm}\in\mF_\gamma^+\cap \mF_\gamma^{K}$  and $ h_\omega \le \frac{e^{-4C'}}{m(\Dom)}$ by construction.

Exactness of the system can now be verified using the same method as detailed in 
\cite{BBMD}. 
\end{proof}
\section{Random coupling} \label{sec:decay1}
\subsection{Estimates on the random recurrence times for the base}
For a single map, the recurrence time of the base with the base gives a key construction parameter for coupling arguments. In the setting of random maps, this recurrence time is $\omega$ dependent. Our first task is to obtain a suitable version of the recurrence time (see $\ell_0$ below, and its use in the following Lemma \ref{lem:eps_ell}). \added[id=WB]{At this stage, it is useful for the reader to recall from section \ref{sec:setup} our assumptions (P1)-P(7); in particular that $a>1$ in (P4). Moreover, recall the regularity class of the equivariant densities defined by the random variable $K_\omega$ from Theorem \ref{thm:acsm}.} 

\begin{lemma}\label{lem:numbertheory} 
Let $N$ and $t_i$ be from the aperiodicity condition (P5).  There is an $\ell_0\in \mathbb N$ so that 
for every $\ell > \ell_0$ there are nonnegative integers $c_i$ such that 
$$ \ell = \sum_{i=1}^N c_i t_i.$$
\end{lemma}
\begin{proof}
See Lemma A2 \cite{Se}. 
\end{proof}

For $\ell\in \N$ define a random variable  $V^\ell:\Omega \to \mathbb R$ by 
$$
V^\ell_\omega=m(\Delta_{\omega, 0}\cap F^{-\ell}_\omega(\Delta_{\sigma^{\ell}\omega, 0})).
$$
Recall that every base $\Delta_{\omega,0} = \Lambda$.

\begin{lemma}\label{lem:eps_ell} For each $\ell > \ell_0$ there is a constant 
$V(\ell) >0$ so that for almost every $\omega \in \Omega$,
$$V_\omega^\ell \geq V(\ell).$$
\end{lemma}

\begin{proof} The result follows from the aperiodicity condition (P5) and bounded distortion (P2). 
First, suppose $F_1=F_\omega^{j_1}: \Lambda \rightarrow \Delta_{\sigma^{j_1}\omega}$
and $F_2=F_{\sigma^{j_1}\omega}^{j_2}: \Lambda \rightarrow \Delta_{\sigma^{j_1+ j_2}\omega}$, satisfying
$$ \frac{m(F_i^{-1} \Lambda \cap  \Lambda)}{m(\Lambda)} \geq \epsilon_i > 0, ~i = 1, 2.$$ 
Then, since $F_i^{-1}$ are bijections when restricted to $\Lambda$, using bounded distortion,
 we get
\begin{equation*}\begin{split}
\frac{m((F_2 \circ F_1)^{-1} \Lambda \cap  \Lambda)}{m(\Lambda)} &\geq 
\frac{m(F_1^{-1}(F_2^{-1} \Lambda \cap  \Lambda)\cap \Lambda)}{m(F_1^{-1} \Lambda \cap  \Lambda)}
\cdot \frac{m(F_1^{-1} \Lambda \cap  \Lambda)}{m(\Lambda)} \\
&\geq \frac{1}{D} \frac{m(F_2^{-1} \Lambda \cap  \Lambda)}{m( \Lambda)}
\cdot \frac{m(F_1^{-1} \Lambda \cap  \Lambda)}{m(\Lambda)} \\
&\geq \frac{1}{D} \epsilon_2 \cdot \epsilon_1. \\
\end{split}
\end{equation*} 
Now, for $\ell > \ell_0$, Lemma \ref{lem:numbertheory} implies $F_\omega^\ell$ can be written as a composition of 
$F_{\omega_i}^{t_i}$ (with at most $\ell$ terms). Iterating the above estimate and using the lower bounds given in condition (P5) implies the existence of the required
lower bound $V(\ell)> 0$. 
\end{proof}

\begin{remark} From the proof of the previous lemma, it is clear that one should not expect 
a lower bound on the $V(\ell)$, 
uniform over all  $\ell$.
\end{remark}

\subsection{Random stopping times} 
Let $\Delta =\{ (\omega, x) | \omega \in \Omega,  x \in \Delta_\omega \}$. 
Denote by $\Delta \otimes_\omega \Delta $ the relative product over $\Omega$, that is
$\Delta \otimes_\omega \Delta=\{ (\omega, x, x') | \omega \in \Omega, x, x' \in \Delta_\omega\}$.
These are measurable subsets of the appropriate product spaces ($\Omega \times \Lambda \times \N$ in the case of $\Delta$, for example), and naturally carry the measures 
$P \times m$  and $\p:= P \times m \times m$ respectively. 
We can lift the tower map $F$ to a product action on $\Delta \otimes_\omega \Delta$
with the property
$F_\omega \times F_\omega :\Delta_\omega\times \Delta_\omega\to \Delta_{\sigma\omega}\times \Delta_{\sigma\omega}$ by applying $F$ in each of the $x,x'$ coordinates.

With respect to this map, we define \emph{auxiliary stopping times} $\tau_1^{\omega} <\tau_2^{\omega} < ... $ to the base  as follows: 

Let $\ell_0$ be the constant given in Lemma \ref{lem:numbertheory}. For $(\omega, x, x') \in \Delta \otimes_\omega \Delta$ set 

\begin{align*}
&\tau_1^\omega(x, x')=\inf\{n\ge \ell_0 \mid  F_\omega^{n}x\in\Delta_{\sigma^n\omega, 0}\};
\\
&\tau_2^\omega(x, x')=\inf\{n\ge \tau_1^\omega(x, x')+\ell_0 \mid  F_\omega^{n}x'\in\Delta_{\sigma^n\omega, 0}\};
\\
&\tau_3^\omega(x, x')=\inf\{n\ge \tau_2^\omega(x, x')+\ell_0 \mid  F_\omega^{n}x\in\Delta_{\sigma^n\omega, 0}\};
\\
&\tau_4^\omega(x, x')=\inf\{n\ge \tau_3^\omega(x, x')+\ell_0 \mid  F_\omega^{n}x'\in\Delta_{\sigma^n\omega, 0}\};
\end{align*}
and so on, with the action alternating between $x$ and $x'$.  Notice that for odd $i$'s the first (resp. for even $i$'s the second) coordinate of $(F_\omega\times F_\omega)^{\tau_i^\omega}(x, x')$ makes a return to $\Delta_{\sigma^{\tau_i^\omega}\omega, 0}$. 

Let $i\ge 2$ be the smallest integer such that $(F_\omega\times F_\omega)^{\tau_i^\omega}(x, x')\in \Delta_{\sigma^{\tau_i^\omega}\omega, 0}\times  \Delta_{\sigma^{\tau_i^\omega}\omega, 0}$. Then we define the \emph{stopping time} $T_\omega$ by
$$T_\omega(x, x'):=\tau_i^\omega(x, x').$$

Next define a sequence of partitions $\xi_1^\omega \prec \xi_2^\omega \prec \xi_3^\omega \prec ... $  of $\Delta_\omega\times\Delta_\omega$  so that $\tau_i^\omega$ is constant on the elements of $\xi_j^\omega$ for all $i\le j,$ $i, j\in\mathbb N$.  Given a partition $\mathcal Q$ of $\Dom$ we write $\mathcal Q(x)$ to denote the element of $\mathcal Q$ containing
$x$.  With this convention, we let

$$
\xi_1^\omega(x, x')=\left(
 \bigvee_{k=0}^{\tau_{1}^\omega-1}F^{-k}_{\omega}\mP_{\sigma^k\omega}
\right)(x)\times \Delta_\omega.
$$

Letting $\pi:\Delta_\omega\times\Delta_\omega\to \Delta_\omega$ be the projection to the first coordinate, we define

$$
\xi_2^\omega(x, x')=\pi\xi_1^\omega(x, x')\times \left(
 \bigvee_{k=0}^{\tau_{2}^\omega-1}F^{-k}_\omega\mP_{\sigma^k\omega}
\right)(x').
$$
Let $\pi'$ be the projection onto the second coordinate. We define $\xi_3^\omega$ by refining the partition on the first coordinate, and so on. If $\xi_{2i}^\omega$ is defined then we define $\xi_{2i+1}^\omega$ by refining each element of $\xi_{2i}^\omega$ in the first coordinate so that $\tau_{2i+1}^\omega$ is constant on each element of $\xi_{2i+1}^\omega$. Similarly $\xi_{2i+2}^\omega$ is  defined by refining each element of $\xi_{2i+1}^\omega$ in the second coordinate so that $\tau_{2i+1}^\omega$ is constant on each new partition element. 
Now we define a partition $\hat\mP_\omega$ of $\Delta_\omega\times \Delta_\omega$ such that $T_\omega$ is constant on its element.  For definiteness suppose that $i$ is even and choose $\Gamma\in\xi_i^\omega$ such that $T_\omega|_{\Gamma}>\tau_{i-1}^\omega$. By construction $\Gamma =A\times B$ such that $F^{\tau_i^\omega}(B)=\Delta_{\sigma^{\tau_i^\omega}, 0}$  and $F^{\tau_i^\omega}A$ is spread around $\Delta_{\sigma^{\tau_i^\omega}\omega}$. We refine $A$ into countably  many pieces and choose those parts which are mapped onto the corresponding base at time $\tau_{i}^\omega$. Note
that $\{T_\omega=\tau_i^\omega\}$ may not be measurable with respect to $\xi_{i}^\omega.$ However,  since $\tau_{i+1}^\omega\ge \ell_0+\tau_{i}^\omega$ and $\xi_{i+1}^\omega$ is defined by dividing $A$ into pieces where $\tau_{i+1}^\omega$ is constant, $\{T_\omega=\tau_i^\omega\}$ is measurable with respect to $\xi_{i+1}^\omega.$ 

\subsection{Tail of the simultaneous  return times}  In this section we estimate the tail of the simultaneous return time $T_\omega$. We start this section with the lower bound on the measure of the set that made return at time $\tom_i$. 
\begin{lemma}\label{lem:return_est1} Let $\lambda_\omega$ and $\lambda'_\omega$ be two probability measures on $\{\Delta_\omega\}$ with densities $\vp, \vp'\in\mF_\gamma^+\cap \mL^{K_\omega}_\infty$. 
 Let $\tilde \lambda = \lambda_\omega \times \lambda_\omega^\prime$. For each $\omega$, for each  $i\ge 2 $ and $\Gamma \in \xi_i^\omega$ such that $T_\omega |_\Gamma > \tau^\omega_{i-1}$ we 
have 
$$
\tilde \lambda\{ T_\omega = \tau^\omega_i |  \Gamma\} \ge C_{\tilde\lambda} V_{\sigma^{\tau^\omega_{i-1}}\omega}^{\tau^\omega_i - \tau^\omega_{i-1}}.
$$
where $0<C_{\bl}<1$. We can fix $C_{\tilde\lambda}= \frac{2D+1}{2(D+1)^2}$, independent of $\bl$, for all $i$ sufficiently large, i.e. $i\ge i_0(\bl)$.
\end{lemma}

\begin{proof}
For definiteness assume $i$ is even. Then $\Gamma\in\xi_{i}^\omega$ has the property $F_\omega^{\tau_{i-1}^\omega}(\pi \Gamma)=\Delta_{\sigma^{\tau_{i-1}^\omega}\omega, 0}$ and $F_\omega^{\tau_{i}^\omega}(\pi' \Gamma)=\Delta_{\sigma^{\tau_{i}^\omega}\omega, 0}$. Together with the definition of $T_\omega$  this implies   $\pi'\{ T_\omega = \tau^\omega_i |  \Gamma\}\cap \Gamma=\pi'\Gamma$. Therefore,  letting $\nu_{\sigma^{\tau_{i-1}^\omega}\omega}:= (F_\omega^{\tau_{i-1}^\omega})_*(\lambda|\pi\Gamma)$ we have 
\begin{align*}
&\tilde \lambda\{ T_\omega = \tom_i |  \Gamma\} =
\frac{\lambda(\pi\{T_\omega=\tom_i\}\cap \pi\Gamma)}{\lambda (\pi\Gamma)}
= \frac{\lambda(F^{-\tom_i}_\omega\Delta_{\sigma^{\tom_i}\omega, 0}\cap F^{-\tom_{i-1}}_\omega\Delta_{\sigma^{\tom_{i-1}}\omega, 0}\cap \pi\Gamma)}
{\lambda(\pi\Gamma)}
\\&= (F_\omega^{\tom_{i-1}})_\ast(\lambda| \pi\Gamma)(F_{\sigma^{\tau_{i-1}^\omega}\omega}^{\tom_{i-1}-\tom_i}\Delta_{\sigma^{\tom_i}\omega, 0}\cap \Delta_{\sigma^{\tom_{i-1}}\omega, 0})=  \nu_{\sigma^{\tau_{i-1}^\omega}\omega}(F_{\sigma^{\tau_{i-1}^\omega}\omega}^{\tom_{i-1}-\tom_i}\Delta_{\sigma^{\tom_i}\omega, 0}\cap \Delta_{\sigma^{\tom_{i-1}}\omega, 0}).
\end{align*}
Finally, item (ii) of Lemma \ref{cbddd} applied to $\nu_{\sigma^{\tau_{i-1}^\omega}\omega}$ implies that 
$$
\tilde \lambda\{ T_\omega = \tom_i |  \Gamma\} \ge \frac{1}{1+D+C_\lambda(D+1)}m(F_{\sigma^{\tau_{i-1}^\omega}\omega}^{\tom_{i-1}-\tom_i}\Delta_{\sigma^{\tom_i}\omega, 0}\cap \Delta_{\sigma^{\tom_{i-1}}\omega, 0}).
$$
Now, the lemma holds with $C_{\tilde\lambda}=\min\{ \frac{1}{(1+C_\lambda)(D+1)}, \frac{1}{(1+C_{\lambda'})(D+1)}\}$.  In view of Remark \ref{rem:indep_of_lambda} we can use 
$C_{\tilde\lambda}= \frac{2D+1}{2(D+1)^2}$ for all  $i$ sufficiently large. 
\end{proof}

The next lemma estimates the distribution of $\tom_i$'s on $\Dom \times \Dom$ by the measure of the tail of 
the random tower. 

\begin{lemma}\label{lem:tailtau} Let $C_{\tilde\lambda}$ be as in Lemma \ref{lem:return_est1}.
For each $\omega$, for each  $i$ and $\Gamma \in \xi_i^\omega$ 
$$\tilde \lambda\{ \tau_{i+1}^\omega - \tau_i^\omega > \ell_0 + n | \Gamma\} 
\leq M_0 MC_{\tilde \lambda}^{-1} \cdot m\{ \hR_{\sigma^{\tom_i + \ell_0}\omega} > n \}.$$
\end{lemma}

\begin{proof}
Suppose that  $i$ is even. Since $\tau_i^\omega$ is
constant on the elements of $\xi_i^\omega$ for every  $\Gamma\in \xi_i^\omega$ we have
$$
\pi\left(\{(x,x')|\hat{R}_\omega\circ F_\omega^{\tau_i^\omega+\ell_0}(x)>n\}\cap\Gamma\right)=
\{x|\hat{R}_\omega\circ F_\omega^{\tau_i^\omega+\ell_0}(x)>n\}\cap\pi\Gamma.
$$
Letting $\nu_{\sigma^{\tau_{i-1}^\omega}\omega}=(F_\omega^{\tau_{i-1}^\omega})_\ast(\lambda |_{}{\pi\Gamma})$,  we have
 \begin{align*}
 & \bl\{\tau_{i+1}^\omega-\tau_i^\omega-\ell_0>n|\Gamma\}
  =\bl\{\hat{R}_{\sigma^{\tom_i+\ell_0}\omega}\circ F_\omega^{\tau_i^\omega+\ell_0}>n|\Gamma\}
  \\  &= \frac{\lambda(\pi\{\hat{R}_{\sigma^{\tom_i+\ell_0}\omega}\circ F_\omega^{\tau_i^\omega+\ell_0}>n\}\cap\pi\Gamma)}{\lambda(\pi\Gamma)}
=(\lambda|\pi\Gamma)\{\hat{R}_{\sigma^{\tom_i+\ell_0}\omega}\circ F_\omega^{\tau_i^\omega+\ell_0}>n\}
\\ & = (F^{\tom_{i-1}}_\omega)_\ast(\lambda|\pi\Gamma)\{\hat{R}_{\sigma^{\tom_i+\ell_0}\omega}\circ F_{\sigma^{\tom_{i-1}}\omega}^{\tom_i-\tom_{i-1}+\ell_0}>n\}=\nu_{\sigma^{\tau_{i-1}^\omega}\omega}\{\hat{R}_{\sigma^{\tom_i+\ell_0}\omega}\circ F_{\sigma^{\tom_{i-1}}\omega}^{\tom_i-\tom_{i-1}+\ell_0}>n\}.
 \end{align*}
Applying item (ii) of Lemma \ref{cbddd}   we obtain 
$$
 \bl\{\tau_{i+1}^\omega-\tau_i^\omega-\ell_0>n|\Gamma\} \le (1+D+C_\lambda(1+D))m\{\hat{R}_{\sigma^{\tom_i+\ell_0}\omega}\circ F_{\sigma^{\tom_{i-1}}\omega}^{\tom_i-\tom_{i-1}+\ell_0}>n\}.
$$
Finally, since the density of $(F_{\sigma^{\tom_{i-1}}\omega}^{\tom_i-\tom_{i-1}+\ell_0})_\ast m$ is bounded above by the first item of  Lemma  \ref{cbddd} we have  
$$
\bl\{\tau_{i+1}^\omega-\tau_i^\omega>\ell_0+n|\Gamma\}
\le M_0M(1+D+C_\lambda(1+D))m\{\hat{R}_{\sigma^{\tom_i+\ell_0}\omega}>n\}.
$$
For $i$ odd the calculation is analogous  and we obtain for all $i$ 
$$
\bl\{\tau_{i+1}^\omega-\tau_i^\omega>\ell_0+n|\Gamma\}
\le M_0MC^{-1}_{\bl}m\{\hat{R}_{\sigma^{\tom_i+\ell_0}\omega}>n\}.
$$
\end{proof}

Suppose we are given a sequence of positive 
integers $\ell_0 \leq  \tau_1 < \tau_2 < \dots <  \tau_n < \dots $ with $\tau_i - \tau_{i-1} \geq \ell_0$ for all $i \geq 2$,  denoted $\vec \tau$,  and a positive integer 
$q>0$. Define associated subsets of $\Delta \otimes_\omega \Delta$ 
$$G_q(\vec{\tau}) = \{ (\omega, x, x') |  \tau_i^\omega(x,x') = \tau_i, ~ i = 1, 2, \dots q \}.$$
This is a {\it partition} into sets where a specified sequence of hitting times up to $q$ is attained: 
$$ \cup_{\vec{\tau}}  G_q( \vec{\tau}) = \cup_{\tau_1< \tau_2< \dots \tau_q} G_q( \vec{\tau})  
= \Delta \otimes_\omega \Delta.$$
For fixed $\omega$ denote $G_q^\omega(\vec{\tau})= G_q(\vec{\tau}) \cap (\Dom \times \Dom),$ 
the cross section of $G_q(\vec{\tau})$ at $\omega$. \added[id=WB]{Let $G_q^\omega=\{(x, x')\in \Dom\times\Dom\mid \tau_j^\omega(x, x')=\tau_j, j=1, \dots, q\}$.}
The following lemma gives a useful estimate on the size of the elements $G_q(\vec{\tau})$.
\begin{lemma}\label{lem:geetau}
There exists a $C>0$ such that for each fixed $\vec{\tau}$, $q>0$
$$\p(G_q(\vec{\tau})) \leq C^{q} \p\{ \tau_1^\omega(x,x') = \tau_1\}\prod_{j=2}^q  \gamma(\tau_j - \tau_{j-1}).$$
where 
$$\gamma(\tau_{i+1}-\tau_{i})=\int_\Omega m\{x\in\Delta_{\sigma^{\tau_{i-1}}\omega}|\,\hat R_{\sigma^{\tau_i+\ell_0}\omega}=\tau_{i+1}-\tau_i-\ell_0\}dP(\omega).$$
\end{lemma}
\begin{proof}
Assume first that $q$ is even and let $G_q^\omega$\replaced[id=WB]{ be as above}{$=\{(x, x')\in \Dom\mid \tau_j^\omega(x, x')=\tau_j, j=1, \dots, q\}$}. 
Let $k_{j}=\tau_j-\tau_{j-1}-\ell_0$, $j=1, \dots, q$. We first show that 
$$
{m\times m(G_q^\omega)}\le 
D {m\times m(G_{q-1}^\omega)}m\times m \{\hR_{\sigma^{\tau_{q-1}+\ell_0}\omega}\circ F_{\sigma^{\tau_{q-2}}\omega}^{\tau_{q-1}-\tau_{q-2}+\ell_0}=k_q\}.
$$
Indeed, for any $\Gamma_{q-1}^\omega\in \xi_{q-1}^\omega$ with $\tau_j^\omega|_{\Gamma_{q-1}^\omega}=\tau_j$ we have 
$$
\begin{aligned}
&\hspace{5mm} \frac{m\times m (G_q^\omega\cap \Gamma_{q-1}^\omega)}{m\times m (\Gamma_{q-1}^\omega)} =
\frac{m\times m (\{\tau_q^\omega=\tau_q\}\cap \Gamma_{q-1}^\omega)}{m\times m ( \Gamma_{q-1}^\omega)}
\\
&=\frac{m (\pi'\{\tau_q^\omega=\tau_q\}\cap \Gamma_{q-1}^\omega)} {m(\pi' \Gamma_{q-1}^\omega)} \le D\frac{m(F_\omega^{\tau_{q-2}}(\pi'\{\tau_q^\omega=\tau_q\}\cap \Gamma_{q-1}^\omega))}{m(\Lambda)}
\\
& \le D m(\Delta_{\sigma^{\tau_{q-2}}\omega, 0}\cap F_\omega^{\tau_{q-2}}\{\hR_{\sigma^{\tau_{q-1}+\ell_0}\omega}\circ F^{\tau_{q-1}+\ell_0}_\omega= k_q \})
\\
& = D m\{\Delta_{\sigma^{\tau_{q-2}}\omega, 0}\cap\hR_{\sigma^{\tau_{q-1}+\ell_0}\omega}\circ F^{\tau_{q-1}-\tau_{q-2}+\ell_0}_{\sigma^{\tau_{q-2}}\omega}= k_q \}
\\
& \le D m\{\hR_{\sigma^{\tau_{q-1}+\ell_0}\omega}\circ F^{\tau_{q-1}-\tau_{q-2}+\ell_0}_{\sigma^{\tau_{q-2}}\omega}= k_q \}.
\end{aligned}
$$
Hence we have 
$$
\frac{m\times m (G_q^\omega\cap \Gamma_{q-1}^\omega)}{m\times m (\Gamma_{q-1}^\omega)} 
\le  D m\{\hR_{\sigma^{\tau_{q-1}+\ell_0}\omega}\circ F^{\tau_{q-1}-\tau_{q-2}+\ell_0}_{\sigma^{\tau_{q-2}}\omega}= k_q\}.
$$
Therefore, 
$$
\begin{aligned}
{m\times m(G_q^\omega)} &=
\sum_{\Gamma_{q-1}^\omega\in\xi_{q-1}^\omega}m\times m(\Gamma_{q-1}^\omega)\frac{m\times m (G_q^\omega\cap \Gamma_{q-1}^\omega)}{m\times m (\Gamma_{q-1}^\omega)}
\\
&\le D {m\times m(G_{q-1}^\omega)} m\{\hR_{\sigma^{\tau_{q-1}+\ell_0}\omega}\circ F^{\tau_{q-1}-\tau_{q-2}+\ell_0}_{\sigma^{\tau_{q-2}}\omega}= k_q \}.
\end{aligned}
$$
By induction, for any $q>2$, we have
$$
{m\times m(G_q^\omega)} \le D^{q-2} {m\times m(G_{1}^\omega)}\prod_{j=2}^q m\{\hR_{\sigma^{\tau_{j-1}+\ell_0}\omega}\circ F^{\tau_{j-1}-\tau_{j-2}+\ell_0}_{\sigma^{\tau_{j-2}}\omega}= k_j \}.
$$
A similar argument applies to obtain the same formula when $q$ is odd. Now by (i) of Lemma \ref{cbddd}, we get
\begin{equation}\label{BS}
{m\times m(G_q^\omega)} \le (DM_0M)^{q-2} {m\times m(G_{1}^\omega)}\prod_{j=2}^q m\{\hR_{\sigma^{\tau_{j-1}+\ell_0}\omega}= k_j \}.
\end{equation}
Notice that, $m\{\hR_{\sigma^{\tau_{j-1}+\ell_0}\omega}= k_{j} \}$ depends only on $\omega_{\tau_{j-1}+\ell_0}, \dots, \omega_{\tau_{j-1}+\ell_0+k_{j}-1}$ while $m\{\hR_{\sigma^{\tau_{j}+\ell_0}\omega}= k_{j+1} \}$ depends only on $\omega_{\tau_{j+\ell_0}}, \dots, \omega_{\tau_j+\ell_0+k_{j+1}-1}$. By definition of $k_j$, we have $\tau_{j-1}+\ell_0+k_{j}-1=\tau_j-1<\tau_j +\ell_0$. Therefore, the product on the right hand side of \eqref{BS} is formed of independent random variables. Moreover, observe that $m\times m(G_1^\omega)$ depends only on $\omega_{\ell_0}, \dots\omega_{\tau_1 -1}$ and $\tau_1 -1 < \tau_1 + \ell_0$. Thus,
$$
\int_{\Omega}m\times m(G_{q}^\omega)dP \le(DM_0M)^{q-2} \int_{\Omega}{m\times m(G_{1}^\omega)}dP\prod_{j=2}^q\int_{\Omega} m\{\hR_{\sigma^{\tau_{j-1}+\ell_0}\omega}= k_j \}dP.$$
Taking $C^q:=(DM_0M)^{q-2}$ gives the desired estimate.
\end{proof}
We now present two lemmas that will be invoked in the proof of Proposition \ref{prop:tail} below.
\begin{lemma}\label{gamma_tail}
 We have $\sum_{\tau_{j+1}-\tau_{j}=K}^\infty\gamma(\tau_{j+1}-\tau_{j})\le C(K)<\infty$. Moreover, $C(K)\to 0$ as $K\to\infty$.
\end{lemma}
\begin{proof}
Using assumption (P7) and Lemma \ref{appendix1} in the appendix, we have
\begin{equation}
\begin{split}
&\sum_{\tau_{j+1}-\tau_{j}=K}^\infty\gamma(\tau_{j+1}-\tau_{j})\le \sum_{\tau_{j+1}-\tau_{j}=K}^\infty C\frac{\log(\tau_{j+1}-\tau_j-\ell_0)^{\hat b}}{(\tau_{j+1}-\tau_j-\ell_0)^{a}}\le \frac{C'\log(K)^{\hat b}}{(K-\ell_0)^{a-1}}:=C(K).
\end{split}
\end{equation}
Moreover; $C(K)\to 0$ as $K\to\infty$. 
\end{proof}
\begin{lemma}\label{tau_tail}
 We have $\sum_{k\ge \ell_0}^\infty\mathbb P\{\tau^{\omega}_1(x,x')=k\} <\infty$. 
\end{lemma}
\begin{proof}
Recall that $\tau^{\omega}_1(x,x')=\ell_0+\hat R_{\sigma^{\ell_0}\omega}\circ F^{\ell_0}_\omega(x)$; i.e., $\tau^{\omega}_1(x,x')$ does not depend on $x'$. Therefore,
\begin{equation*}
\begin{split}
&\sum_{k\ge \ell_0}^\infty\mathbb P\{\tau^{\omega}_1(x,x')=k\}=\sum_{k\ge 0}^\infty P\times m\{\hat R_{\sigma^{\ell_0}\omega}\circ F^{\ell_0}_\omega(x)=k\}\times m(\Delta_\omega)\\
&\le M\sum_{k\ge 0}^\infty P\times m\{\hat R_{\sigma^{\ell_0}\omega}\circ F^{\ell_0}_\omega(x)=k\}\le M^2M_0\sum_{k\ge 0}^\infty P\times m\{\hat R_{\omega}=k\}\\
&\le M^2M_0C\sum_{k\ge 0}\frac{(\log k)^{\hat b}}{k^a}<\infty,
\end{split}
\end{equation*}
where we have used the first item of Lemma \ref{cbddd} and (P7).
\end{proof}
We can now present the main result of this section.
\begin{proposition}\label{prop:tail} Let $\delta>0$ be given. Let $\lambda_\omega$ and $\lambda_\omega'$ be two \added[id=WB]{two families of} probability measures on $\{\Delta_\omega\}$ with densities $\vp, \vp'\in\mF_\gamma^+\cap \mL^{K_\omega}_\infty$. Let $\bl := \lambda_\omega \times\lambda_\omega'$. 
Then there exists a constant $\hat C_{\bl}$ and a subset $\Omega_5\subset \Omega$ of full measure and a random variable  $n_5(\omega)$ on $\Omega_5$ such that for any 
$n > n_5$ the following holds
$$ \bl \{ T_\omega > n \} \leq  \hat C_{\bl}\frac{(\log n)^b}{n^{a-1-\delta}}.$$
Moreover, there exist $C>0$,  $u'>0, 0<v'<1$ such that for any $n$ 
$$P\{\omega\mid  n_5(\omega) > n \} \leq C e^{-u'n^{v'}}.$$
\end{proposition}
\begin{proof}
Let  $c:=\min\{ \frac{\delta}{a+1}, 1/2\}$.  For a.e. $\omega\in\Omega$ we have\footnote{\label{justfood} Notice that we have chosen $q=n^c$ to keep the proof and the estimates of $Y_1$, $Y_2$ and $C_\omega$ as simple as possible. One may try $q=(\log n)^d$ for sufficiently large $d$ so that $Y_2$ decays faster than $Y_1$ and $C_\omega$ remains integrable and to get a quenched decay rate of the form $\frac{(\log n)^{b+d'}}{n^{a-1}}$, for some $d'\ge d$, in Theorem \ref{thm:DC}. However, no matter how we choose $q=g(n)$, with $g(n)\to\infty$ as $n\to\infty$, a quenched correlation decay rate of the form $\frac{(\log n)^b}{n^{a-1}}$, which is analogous to what one expects in the deterministic setting, cannot be achieved since we want to get information on the integrability of the $C_{\omega}$ in Theorem \ref{thm:DC}. The shift of the Lipschitz constant $K_\omega$, and hence the dependence of that constant on $n$, in equation \eqref{eq:corr_tpsi} and the non-uniformity of the tail in (P4) are the main reasons for getting a rate at the order $\frac{1}{n^{a-1+\delta}}$, for any $\delta>0$.}
\begin{equation}\begin{aligned}
\bl\{T_\omega >n\}&\le
 \sum_{ i< \spltp}\bl\{T_\omega>n; \tau_{i-1}^\omega\le n< \tau_i^\omega\}
+ \bl\{T_\omega>n; \tau_{\spltp}^\omega\le n\}\\
 &\leq \sum_{ i< \spltp}\bl\{ \tau_{i-1}^\omega\le n< \tau_i^\omega\}
 + \bl\{T_\omega>n; \tau_{\spltp}^\omega\le n\}\\
 &=: Y_1+Y_2.
\end{aligned}
\end{equation}
We will show that the term $Y_1$ decays at the indicated log-polynomial rate (in $n$) while the term $Y_2$ decays as stretched exponential, which implies the result. First, for the term $Y_1$ we have: 
\begin{equation}\begin{split}\label{eq:Y1expansion1}
\sum_{ i< \spltp}\bl\{\tau_{i-1}^\omega\le  n< \tau_i^\omega\} =
&\sum_{ i< \spltp}\sum_{\Gamma\in \xi_{i-1}^\omega}\bl\{\tau_{i-1}\le n <\tom_i\mid_{}\Gamma\}\bl(\Gamma) \\
&=\sum_{ i< \spltp}\sum_{\substack{\Gamma\in \xi_{i-1}^\omega \\ \tau_{i-1}|\Gamma\le n}}\bl\{\tau_{i-1}\le n <\tom_i\mid_{}\Gamma\}\bl(\Gamma)\\
&\le \sum_{ i< \spltp}\sum_{\substack{\Gamma\in \xi_{i-1}^\omega \\ \tau_{i-1}|\Gamma\le n}}\sum_{j=1}^i \bl\{\tau_j^\omega -  \tau_{j-1}^\omega \geq \frac{n}{i}\mid\Gamma\}\bl(\Gamma).
\end{split}
\end{equation}
For each term in the sum \eqref{eq:Y1expansion1}, using  Lemma 
\ref{lem:tailtau} we obtain,
\begin{equation} \begin{split}\label{eq:Y1expansion2}
\bl\{\tau_j^\omega -  \tau_{j-1}^\omega &\geq \frac{n}{i} \mid\Gamma\}\bl(\Gamma)
= \bl\{\tau_j^\omega -  \tau_{j-1}^\omega \geq (\frac{n}{i} - \ell_0) + \ell_0 | \Gamma\}\bl(\Gamma)\\
& \le MM_0C_{\bl}^{-1}m\{{\hat R}_{\sigma^{\tau_{j-1}^\omega  + \ell_0}\omega} > \frac{n}{i} - \ell_0\} \bl( \Gamma )\\
& \le MM_0C_{\bl}^{-1}\bl( \Gamma )\sum_{k > \frac{n}{i} - \ell_0}
m\{x \in \Lambda | R_{\sigma^{\tau_{j-1}^\omega + \ell_0 - k}\omega} > k \}.
\end{split}
\end{equation}
For each $\omega \in \cap_{n \in \Z}\, \sigma^{-n}( \Omega_1)$, where $\Omega _1$ is the full measure subset from condition (P4), we want to define a random variable $n_4(\omega)$ such that for any $n \ge n_4(\omega)$ we have $n_1(\sigma^{\tau_{j-1}^\omega+ \ell_0-k}\omega) \le \lfloor n^{1-c}\rfloor
$ for any  $k\ge \frac{n}{i} -\ell_0$, $i=1, \dots,  \lfloor n^{1-c}\rfloor$,  so that we can apply the uniform decay rates from (P4).  Below the constraint $\tom_{i-1}|\Gamma \le n$ is  crucial. 
$$n_4(\omega) = 
\inf\{ m | \forall n > m, \forall N \in \{1, 2,  \dots n + \ell_0\}, \forall k > \lfloor{n^{1-c}}\rfloor - \ell_0, n_1(\sigma^{N-k}\omega) < k\}.$$
We claim that $n_4$ has a stretched exponential tail. 
\begin{equation*}\begin{split}
P\{ n_4(\omega) >m\} &\leq \sum_{n>m} \sum_{N=1}^{n+\ell_0} \sum_{k\ge\lfloor n^{1-c}\rfloor- \ell_0}P\{ n_1( \sigma^{N-k} \omega ) >k\}\\ 
&=  \sum_{n>m} \sum_{N=1}^{n+\ell_0} \sum_{k\ge\lfloor n^{1-c}\rfloor- \ell_0}P\{ n_1( \omega ) >k\}\\
&\leq\sum_{n>m} (n+ \ell_0) e^{-u'n^{v'}} \leq e^{-u''m^{v''}},
\end{split}
\end{equation*}
for an appropriate choice of $u''>0$, $0<v''<1$. 
Now, for $n > n_4$ using the fact that $\tau_{j-1}^\omega  \leq n$ and Lemma  \ref{appendix1} we can further upper bound the sum in the 
equation  \eqref{eq:Y1expansion2} by 
\begin{equation}\label{termin75}
\sum_{k\ge \frac{n}{i}-\ell_0}m\{x \in \Lambda | R_{\sigma^{\tau_{j-1}^\omega + \ell_0 - k}\omega} > k \}
\leq  C \frac{[ \log(\frac{n}{i} - \ell_0)]^b}{[\frac{n}{i} - \ell_0]^{a-1}} \le C'i^a \frac{[ \log n]^b}{ n^{a-1}}.
\end{equation}
Now, inserting the  estimate \eqref{termin75} back into equation \eqref{eq:Y1expansion2} and 
substituting that result into  \eqref{eq:Y1expansion1}  we obtain the final estimate on 
$Y_1$:
\begin{equation*} 
Y_1 \leq M^2M_0 C_{\bl}^{-1} K \hat C \sum_{ i< \spltp} i^a \frac{[\log n ]^b}{n^{a-1}}
\leq C' n^{c(a+1)}\frac{[\log n ]^b}{n^{a-1}} \leq C'\frac{[\log n ]^b}{n^{a-1-\delta}}.
\end{equation*}

Now we tackle the term $Y_2$ by decomposing $\Delta \otimes_\omega \Delta$ into two pieces. 
First for parameters $K>0$ and $0<\rho<1$, and integer $q >0$ define
$$B_q(K, \rho) = \{(\omega, x, x') | \#\{ i | 2 \leq i \leq q, \tau_i^\omega - \tau_{i-1}^\omega >K\} 
> \rho q\}.$$
We are going to pick the parameters $K$, and $\rho \sim 1$ later, but the idea is that for points in $B_q(K,\rho)$ the first $q$ return times have many (at least $\rho q$) large (bigger than $K$) gaps.  
Our decomposition will be according to this $B_q(K, \rho)$ for $q = \lfloor n^c \rfloor$:
\begin{equation}\begin{split}
&Y_2 = \bl\{ T_\omega > n, \tau_q < n\} = \bl(\{ T_\omega > n\} \cap B_q(K, \rho))
\\&+ \bl(\{ T_\omega > n \} \cap [B_q(K, \rho)]^c)
\leq \bl( B_q(K, \rho)) + \bl(\{ T_\omega > n \} \cap [B_q(K, \rho)]^c).\\
\end{split}
\end{equation}
In order to estimate the first term in this expression,
fix a sequence of integers
$2 \leq t_1 < t_2 <  \dots t_s$ for $\rho q \leq s \leq q-1$ and define
$$B_q(K, \{ t_i\}) = \{(\omega, x, x') | \tau_{t_i}^\omega - \tau_{t_i -1}^\omega >K, i = 1, 2, \dots s\}.$$
Then $B_q(K,c) = \cup_{s=\rho q}^{q-1} \cup_{t_1 < \dots < t_s} B_q(K, \{t_i\})$ and by Lemma \ref{lem:geetau} we can estimate measures of the terms on the right by
\begin{equation}
\begin{split} 
&\p( B_q(K, \{ t_i\})) = \sum_{\tau_1 < \tau_2 < \dots \tau_q} \p( B_q(K, \{ t_i\}) \cap G_q(\vec{\tau}))\\
&\leq 
\underset{\begin{subarray}{c}
  \tau_1 < \tau_2 < \dots \tau_q \\
\tau_{t_i} - \tau_{t_i -1} >K\\
i = 1, 2, \dots s
  \end{subarray}}
  {\sum}\p(G_q(\vec{\tau}))
\leq C^{q} \sum_{\tau_1} \p\{ \tau_1^\omega(x,x') = \tau_1\} 
\underset{\begin{subarray}{c}
  \tau_2 < \tau_3 < \dots \tau_q \\
\tau_{t_i} - \tau_{t_i -1} >K\\
i = 1, 2, \dots s
  \end{subarray}}
  {\sum} \prod_{j=2}^q  \gamma (\tau_j - \tau_{j-1}) \\
&=C^{q} \sum_{\tau_1}  \p\{ \tau_1^\omega(x,x') = \tau_1\} 
\prod_{i=1}^s \sum_{\tau_{t_i }- \tau_{t_i-1}} \gamma (\tau_{t_i} - \tau_{t_i-1})\prod_{j\neq t_i} \sum_{t_j - t_{j-1}}\gamma (\tau_j - \tau_{j-1})\\
&\leq C^{q} \sum_{m\geq \ell_0} \p\{ \tau_1^\omega(x,x') = m\} 
\left(\sum_{m>K} \gamma (m)\right)^s \left(\sum_{m\geq \ell_0} \gamma (m)\right)^{q-s} .\\
\end{split}
\end{equation}
Now applying Lemmas \ref{tau_tail} and \ref{gamma_tail} we obtain, assuming $\rho > \frac{1}{2}$,
\begin{equation}
\begin{split}
\p (B_q(K,\rho))& = \sum_{s=\rho q}^{q-1} \sum_{t_1 < \dots < t_s} \p(B_q(K, \{t_i\}))
\leq \sum_{s=\rho q}^{q-1} [2\hat C [C(K)]^\rho]^{q}<[2e\hat C [C(K)]^\rho]^{q}.
\end{split}
\end{equation}
We pick $K$ large enough so that 
$$ 2e { \hat{C}} [C(K)]^\rho
:= \kappa_1<1.$$
This shows $\p(B_q(K,\rho)) \leq \kappa_1^q$. Since we want estimates over 
individual fibres $\Dom \times \Dom$ we finally observe the above estimate
shows $m \times m( B_q^\omega(K,\rho)) \leq \kappa_1^{q/2}$ except on 
a set of $\omega$ of measure at most $\kappa_1^{q/2}$.  Once again, an application of Borel-Cantelli shows
there is a full measure set $\Omega_5 \subseteq \Omega$ and $n_5(\omega) \geq n_4(\omega)$ with stretched 
exponential tails (there exists $u>0$ and $0<v<1$ so that 
$ P\{ n_5 > n\} \leq e^{-un^v}$ and such that, for every $\omega \in \Omega_5$ and $n > n_5(\omega) $, 
$m \times m( B_q^\omega(K,\rho)) \leq \kappa_1^{q/2}$. 

We now turn our attention to the complement of $B_q(K, \rho)$.  Note that for each $\omega, \vec{\tau}= (\tau_1, \dots \tau_q)$ either $G_q^\omega(\vec{\tau}) \subseteq  B_q^\omega(K, \rho)$ or 
$G_q^\omega(\vec{\tau}) \cap B_q^\omega(K, \rho) = \emptyset$.  Let us call those 
$\vec{\tau}$ in the former class  $\omega, q- \,\textnormal{good}$. The others we will call $\omega, q-\,\textnormal{bad}$. Therefore, for $\omega$ fixed
\begin{equation} \begin{split}
\bl(\{ T_\omega >n; \tau^\omega_q <n\}&\cap [ B_q^\omega(K, \rho)]^c)\leq 
\bl(\{ T_\omega > \tau^\omega_q\}\cap [ B_q^\omega(K, \rho)]^c)\\
&= \sum_{\vec{\tau}: ~ \omega, q- \,\textnormal{bad}} \bl( \{T_\omega > \tau^\omega_q\} \cap G_q^\omega(\vec{\tau})).
\end{split}
\end{equation}
We move now to estimate the individual terms in the sum over $\omega, q- \,\textnormal{bad}$ terms. Note that each $G_q^\omega(\vec{\tau})$ is $\xi^\omega_q$ measurable. Therefore we can write
$$G_q^\omega(\vec{\tau}) = \cup_{\Gamma_q \in \xi^\omega_q, \, \Gamma_q 
\subseteq G_q^\omega(\vec{\tau})} \Gamma_q$$ as a disjoint union. Recall that $\{T_w > \tau_{q-1}\}$ is measurable with respect to $\xi_q^\omega$. Therefore, for each 
$\Gamma_q$ in the above decomposition, either $\Gamma_q \cap \{T_w > \tau_{q-1}\} = \Gamma_q$
or $\Gamma_q \cap \{T_w > \tau_{q-1}\} = \emptyset$. Call the former 
$\omega, q- \,\textnormal{good}$ and the latter $\omega, q- \,\textnormal{bad}$.  Finally, note that if
$\Gamma_q$ is $\omega, q- \,\textnormal{bad}$ then $\Gamma_q \cap \{T_w > \tau_{q}\} = \emptyset$.
Now we estimate:

\begin{equation*}
\begin{split}
&\bl (\{ T_w > \tau_q\} \cap G_q^\omega(\vec{\tau}))= \sum_{\Gamma_q:~ \omega, q- \,\textnormal{good}} \bl(\{ T_w > \tau_q\} \cap \Gamma_q)\\
&= \sum_{\Gamma_q:~ \omega, q- \,\textnormal{good}} \bl(\{ T_w > \tau_q\} | \{ T_w > \tau_{q-1}\} \cap \Gamma_q)\bl(\{ T_w > \tau_{q-1}\} \cap \Gamma_q)\\
&\leq \sum_{\Gamma_q:~ \omega, q- \,\textnormal{good}} 
(1 - C_{\bl}V_{\sigma^{\tau_{q-1}} \omega }^{\tau_q - \tau_{q-1}})
\bl(\{ T_w > \tau_{q-1}\} | \{ T_w > \tau_{q-2}\} \cap \Gamma_q)\bl(\{ T_w > \tau_{q-2}\} \cap \Gamma_q)\\
&\dots \\
&\leq \sum_{\Gamma_q:~ \omega, q- \,\textnormal{good}} \Pi_{j=2}^q  (1 - C_{\bl}V_{\sigma^{\tau_{j-1} \omega}}^{\tau_j - \tau_{j-1}})
\bl(\{ T_w > \tau_{1}\} \cap \Gamma_q).
\end{split}\end{equation*}
Now, since each good $\Gamma_q$ in the above sum is a subset of $G_q^\omega(\vec{\tau})$
that is $\omega, q- \,\textnormal{bad}$
we know that $\#\{ i | 2 \leq i \leq q, \tau_i - \tau_{i-1} \leq K \} > (1-\rho)q$. Therefore, in the above product, 
considering only those factors in the product, and keeping in mind the lower bound 
given by Lemma \ref{lem:eps_ell} we get 
\begin{align*}
\bl (\{ T_w > \tau_q\} \cap G_q^\omega(\vec{\tau})) &\leq (1 - C_{\bl}V(K))^{(1-\rho)q} 
\sum_{\Gamma_q:~ \omega, q- \,\textnormal{good}}\bl(\{ T_w > \tau_{1}\} \cap \Gamma_q)\\
&\leq (1 - C_{\bl}V(K))^{(1-\rho)q} 
\sum_{\Gamma_q:~ \omega, q- \,\textnormal{good}}\bl(\Gamma_q).
\end{align*}
Finally, summing first over all the good $\Gamma_q$ and  \replaced[id=WB]{then over all $G_q^\omega(\vec{\tau})$ for
$\omega, q- \,\textnormal{bad } \vec\tau$}{then over all the
$\omega, q- \,\textnormal{bad}$ sets $G_q^\omega(\vec{\tau})$} we obtain
\begin{equation*}
\bl (\{ T_w > \tau_q\} \cap [ B_q^\omega(K, \rho)]^c) \leq (1 - C_{\bl}V(K))^{(1-\rho)q}.
\end{equation*}
Set $\kappa = \max\{\kappa_1, (1 - C_{\bl}V(K))^{(1-\rho)}\} <1$ and obtain 
$$Y_2 \leq C'' \kappa^q = C'' \kappa^{\lfloor n^c \rfloor}$$
for all $n > n_5(\omega)$, giving the claimed stretched exponential decay. This completes the proof of the lemma. 
\end{proof}

\subsection{Coupling}
Here we consider $\hF_\omega=(F_\omega\times F_\omega)^{T_\omega}$ which is a mapping from  $\hDom=\Dom\times \Dom$ into $\hat\Delta_{\sigma^{T_\omega}\omega}$.  
Let $\hat\xi_{1}^\omega$ be the partition of $\hDom$ on which $T_{\omega}$ is constant. 
 Let $T_{1, \omega}<T_{2, \omega} \dots$ be stopping times  on $\hDom$  defined as
$$
T_{1, \omega}=T_\omega, \quad T_{n, \omega}=T_{n-1, \omega}+ T_{\sigma^{T_{n-1, \omega}}\omega}\circ \hF^{n-1}_\omega.
$$
For $u, z\in\hDom$ we define a separation time $\hat s(u, z)$  associated with $\hF_\omega$ as the smallest $n\ge 0$ such that $\hF_\omega^n(u)$ and $\hF_\omega^n(z)$ lie in distinct elements of $\hat\xi_1^{\sigma^{T_{n, \omega}}\omega}$ \footnote{Notice that,  for any $u=(x, x'), z=(y, y')\in\Dom$ if $\hat s(u, z)>n$ then $s(x, x')>n$ and $s(y, y')>n$.}. 

Let $\lambda_\omega$ and $\lambda_\omega'$ be two probability measures on $\{\Dom\}$ with densities $\vp, \vp'\in\mF_\gamma^+\cap \mL^{K_\omega}_\infty$. 
Let $\bl=\lambda_\omega\times \lambda_\omega'$ and  $\Phi=d\bl/d(m\times m)$, then $\Phi(x, x')=\vp(x)\vp'(x')$.  The next lemma  establishes the regularity of $\hF_\omega$ and $\Phi$. 

\begin{lemma}\label{the_lost_lemma}

\begin{enumerate}
\item For any $n>0$, $u, z\in\hDom$ with $\hat s(u, z)\ge n$ 
$$
\left|\log\frac{J\hF^n_\omega(u)}{J\hF^n_\omega(z)}\right|\le \hD\gamma^{\hat s(\hF^n_\omega u, \hF^n_\omega z)},
$$
where $\hD\ge2D$ is a constant. 
\item For any $n>0$, $u, z\in\hDom$  
$$
\left|\log\frac{\Phi(u)}{\Phi(z)} \right|\le C_{\Phi}\gamma^{\hat s(u, z)},
$$
where $C_{\Phi}=C_\vp+C_{\vp'}$.  
\end{enumerate}
\end{lemma} 
 \begin{proof}
Let $u=(x, x')$, $z=(y, y')$. For $n>0$ choose $k$ so that $\hF^n_\omega(u)=(F_\omega\times F_\omega)^k(u)$.  \added[id=WB]{Then} 
\begin{align*}
&\left|\log\frac{J\hF^n_\omega(x, x')}{J\hF^n_\omega(y, y')}\right|\le \left|\log\frac{JF^k_\omega(x)}{JF^k_\omega(y)}\right|+
\left|\log\frac{JF^k_\omega(x')}{JF^k_\omega(y')}\right|\\
&\le D\gamma^{s(F^k_\omega x, F^k_\omega y)}+ D\gamma^{s(F^k_\omega x', F^k_\omega y')} \le \hD\gamma^{\hat s(\hF^n_\omega u, \hF^n_\omega z)},
\end{align*}
where we have used $\hat s(u,z)\le\min\{s(x,x'), s(y,y')\}$. Similarly for the second item we have 
$$
\left|\log\frac{\Phi(x, x')}{\Phi(y, y')} \right|  \le 
\left|\log\frac{\vp(x)}{\vp(y)} \right| +\left|\log\frac{\vp'(x')}{\vp'(y')}\right| 
\le C_{\Phi}\gamma^{\hat s(u, z)}.
$$
\end{proof}

Let $\hat\xi_{i}^\omega$ be the partition of $\hDom$ on which $T_{1, \omega}, \dots, T_{i, \omega}$ are constant. For $z\in \hDom$ let $\hat\xi_{i}^\omega(z)$ be the element containing $z$.  Given $\Phi(x, x')=\vp(x)\vp(x')$ let $i_1(\Phi)$ be such that $C_\Phi\gamma^{i_1}< \hD$. For $i<i_1$ let $\hat\Phi_i\equiv \Phi.$ For $i\ge i_1$, let 
\begin{equation}\label{eq:hatPhi}
\hat \Phi_i(z) =\left[\frac{\hat \Phi_{i-1}(z)}{J\hF^{i}_\omega(z)}- \eps \min_{u\in \hat\xi_i^\omega(z)}\frac{\hat \Phi_{i-1}(u)}{J\hF^{i}_\omega(u)}\right]J\hF^i_\omega(z),
\end{equation}
where $\eps$ is a small number that will be defined below.
Since $(\hat \Phi_i - \hat \Phi_{i-1})/{J\hF^i_\omega}$ is constant  on every $\Gamma\in\hat\xi_i^\omega$,  we have 
\begin{equation*}
\pi_\ast(\hF^i_\omega)_\ast((\hat\Phi_{i-1}-\hat\Phi_i)(m\times m)|\Gamma)=\pi_\ast'(\hF^i_\omega)_\ast((\hat\Phi_{i-1}-\hat\Phi_i)(m\times m)|\Gamma).
\end{equation*}
Note that,  $\hat \Phi_i$ is the density of the part of $\bl$ which has not been matched up to time $T_{i, \omega}$.
 \begin{lemma}\label{phii}
For all sufficiently small $\eps>0$ in \eqref{eq:hatPhi}, there exists $0<\eps_1<1$ independent of 
$\Phi$ such that for almost every $\omega$ and for all $i\ge i_1$ 
$$
\hat \Phi_i\le (1-\eps_1)\hat\Phi_{i-1} \quad \text{on} \quad \hDom.
$$ 
\end{lemma}
We will introduce the following densities in order to prove Lemma \ref{phii}. 
For $z\in\hDom$ let 
$$
\tilde\Psi_{i_1-1}(z)=\frac{ \Phi(z)}{J\hF^{i_1-1}_\omega(z)}
$$
and for $i\ge i_1$, let 
$$
\Psi_{i}(z)=\frac{\tilde\Psi_{i-1}(z)}{J\hF_{\sigma^{T_{i-1}, \omega}\omega}(\hF^{i-1}_\omega(z))},  \quad
\eps_{i, z}=\eps\cdot \min_{u\in\hat \xi_i(z)}\Psi_{i}(u), \quad \tilde \Psi_i(z)=\Psi_i(z)-\eps_{i, z}.
$$
Lemma \ref{phii} then follows from  the following lemma:
\begin{lemma}
There exists $\hat C$ such that for all sufficiently small $\eps$ the following holds: for any $z\in\hDom$ with $u\in \hat\xi_i^\omega(z)$ and $i\ge i_1$
$$
\left|
\log \frac{\tilde \Psi_i(u)}{\tilde \Psi_i(z)} \right| \le \hat C\gamma^{\hat s(\hF^i_\omega u, \hF^i_\omega z)}.
$$
\end{lemma}
\begin{proof}
By definition of $\Psi_i$ and item (1) Lemma \ref{the_lost_lemma} of  we have 
\begin{equation}\label{Psii1}
\begin{aligned}
\left|\log \frac{\Psi_i(u)}{\Psi_i(z)}\right| \le \left|\log \frac{\tilde \Psi_{i-1}(u)}{\tilde \Psi_{i-1}(z)}\right|+\hD\gamma^{\hat s(\hF^i_\omega z, \hF^i_\omega u)}.
\end{aligned}
\end{equation}

 Since $\eps_{i, z}$ is constant on $\hat\xi_i^\omega(z)$ we let $\eps_i=\eps_{i, z}$. We have 
 \begin{equation*}\begin{aligned}
 \left|\log \frac{\tilde\Psi_i(u)}{\tilde\Psi_i(z)} - \log \frac{\Psi_i(u)}{\Psi_i(z)}\right| = \left|\log \frac{\Psi_i(u)-\eps_i}{\Psi_i(u)}\frac{\Psi_i(z)}{\Psi_i(z)-\eps_i}\right| \le \left|\frac{\frac{\eps_i}{\Psi_i(z)}-\frac{\eps_i}{\Psi_i(u)}}{1-\frac{\eps_i}{\Psi_i(z)}}\right| \\
\le \frac{\eps_i}{\Psi_i(z)}\left|\frac{\Psi_i(u)}{\Psi_i(z)}-1\right|\frac{1}{1-\eps}\le \frac{\eps}{1-\eps}C\left|\log \frac{\Psi_i(u)}{\Psi_i(z)}\right|.
 \end{aligned}\end{equation*}
Notice that $C$ in the latter inequality increases as $\frac{\Psi_i(u)}{\Psi_i(z)}$ increases. Allowing $\eps$  to depend on $i$, $z,$ $u$ and $\omega$,\footnote{Notice that when $i=i_1$ we can choose $\eps$ uniform for all $u$ and $z$, allowing dependence only on $\hD$ and $C_\Phi$.} for a given $0< \eps'<\gamma^{-1}-1$ we can choose $\eps$ small enough so that  
\begin{equation}\label{choice_of_eps}
\frac{\eps}{1-\eps}C< \eps'
\end{equation}
we obtain 
\begin{equation}\label{Psii2}
\left|\log \frac{\tilde\Psi_i(u)}{\tilde\Psi_i(z)}\right| \le (1+\eps') 
\left|\log \frac{\Psi_i(u)}{\Psi_i(z)}\right|.
\end{equation}
By \eqref{Psii1} and \eqref{Psii2} we  obtain 
\begin{equation}\label{recursive}
\left|\log \frac{\tilde\Psi_i(u)}{\tilde\Psi_i(z)}\right| \le (1+\eps')\left(\left|\log \frac{\tilde\Psi_{i-1}(u)}{\tilde\Psi_{i-1}(z)}\right|+ \hD\gamma^{\hat s(\hF^i_\omega u, \hF^i_\omega z)}\right).
\end{equation}
Moreover, for $i=i_1$ we have 
\begin{align*}
\left|\log \frac{\tilde\Psi_{i_1}(u)}{\tilde\Psi_{i_1}(z)}\right| 
&\le (1+\eps')\left(\left|\log \frac{\Phi(u)}{\Phi(z)}\right|+ \hD\gamma^{\hat s(\hF^i_\omega u, \hF^i_\omega z)}\right)
\\&\le (1+\eps')(C_\Phi\gamma^{\hat s(u, z)}+\hD\gamma^{\hat s(\hF^{i_1}_\omega u, \hF^{i_1}_\omega z)})
\\&\le (1+\eps')2\hD\gamma^{\hat s(\hF^{i_1}_\omega u, \hF^{i_1}_\omega z)}.
\end{align*}
Note that in the last inequality we have used $C_\Phi\gamma^{\hat s(u, z)}\le \hD\gamma^{\hat s(\hF^{i_1}_\omega u, \hF^{i_1}_\omega z)}$. 
Finally,  using the relation $\hat s(\hF^{i-j}_\omega u, \hF^{i-j}_\omega z)=\hat s(\hF^{i}_\omega u, \hF^{i}_\omega z)+j$ we have 
\begin{equation}\label{4}
\left|\log\frac{\tilde \Psi_i(u)}{\tilde \Psi_i(z)}\right|\le \hat C \gamma^{s(\hF^{i}_\omega u, \hF^{i}_\omega z))},
\end{equation}
where $\hat C=2(1+\eps')\hD\sum_{j=0}^\infty[(1+\eps')\gamma]^j$.

Now, we show by an inductive argument that $\eps$ in \eqref{choice_of_eps} can be chosen independent of $i$, $u$, $z$, $\omega$. 
First of all notice that we can choose $\eps$ independent of $u,$ $z$ for $i=i_1$ because 
$$
\frac{\Psi_{i_1}(u)}{\Psi_{i_1}(z)}=\frac{\Phi(u)}{\Phi(z)}\frac{J\hF^{i_1}_\omega(z)}{J\hF^{i_1}_\omega(u)} \le (1+\hD)^2.
$$
Let  $j>i_1$ and suppose that $\eps$ is small enough so that  \eqref{4} holds for all $i<j$ and $u\in\hat \xi_i^\omega(z)$. 
Then  by \eqref{Psii1} we have 
$$
\left|\log \frac{\Psi_{i}(u)}{\Psi_{i}(z)}\right| \le \hat C+\hD,
$$
which implies that $\frac{\Psi_{i}(u)}{\Psi_{i}(z)}\in [e^{-(\hat C+\hD)}, e^{\hat C+\hD}]$. Therefore $C$ in \eqref{choice_of_eps}
 is bounded by $e^{\hat C+\hD}$. Hence by choosing $\eps< \eps'e^{-(\hat C+\hD)}$ we conclude that the estimate in \eqref{recursive} holds for $i=j$.  
 \end{proof}
\begin{lemma}\label{lem:main}
Let $0<\eps_1<1$ be as in Lemma \ref{phii}. For almost every $\omega$ and all $n\in \mathbb N$ 
$$
|(F_\omega^n)_*(\lambda)-(F_\omega^n)_*(\lambda')|\le 2\bl\{T_{i_1, \omega}>n\}+2\sum_{i=i_1}^\infty(1-\eps_1)^{i-i_1+1}\bl\{T_{i, \omega}\le n <T_{i+1, \omega}\}, 
$$
where $\bl=\lambda_\omega\times \lambda_\omega'$.
\end{lemma}
\begin{proof}
In Lemma \ref{phii} the estimates for the mass of $\bl$  after the $i^{th}$ iterate matching  was given. Now we will relate that estimate to the iterates of $F_\omega$.  Define  $\Phi_0$, $\Phi_1, \dots$  as follows: for $z\in \Dom\times\Dom$ let 
$$
\Phi_n(z)=\hat\Phi_i(z) \quad \text{when} \quad T_{i, \omega}(z)\le n< T_{i+1, \omega}(z),
$$
where $\hat\Phi_i(z)$ is as in \eqref{eq:hatPhi}.
We first prove that $|(F_\omega^n)_*(\lambda)-(F_\omega^n)_*(\lambda')| \le 2 \int \Phi_nd(m\times m)$. 
Below we use  the notation $\Phi (m\times m)$  to denote a measure whose density with respect to $m\times m$ is $\Phi$. 
First of all recall that $\Phi_0=\Phi$ and  write $\Phi=\Phi_n+\sum_{k=1}^n(\Phi_{k-1}-\Phi_k)$. We have 
\begin{equation}\label{eq:matching}
\begin{aligned}
|(F_\omega^n)_*(\lambda) &-(F_\omega^n)_*(\lambda')| \\
&= |\pi_*(F_\omega\times F_\omega)^n_*(\Phi(m\times m))-\pi_*'(F_\omega\times F_\omega)^n_*(\Phi(m\times m))| 
\\& \le|\pi_*(F_\omega\times F_\omega)^n_*(\Phi_n(m\times m))-\pi_*'(F_\omega\times F_\omega)^n_*(\Phi_n(m\times m))| 
\\&+\sum_{k=1}^n|(\pi_*-\pi_*')(F_\omega\times F_\omega)^n_*((\Phi_{k-1}-\Phi_{k})(m\times m))|.
\end{aligned}\end{equation}
Since, for any $A\subset \Delta_{\sigma^n\omega}$  we have  
$$
\pi_*(F_\omega\times F_\omega)^n_*(\Phi_n(m\times m))(A)= \int_{F^{-n}_\omega(A)\times \Dom}\Phi_nd(m\times m),
$$ 
the first term in the final sum in \eqref{eq:matching} is bounded by $2\int{\Phi_n}d(m\times m)$. Now, we claim that  all other terms in \eqref{eq:matching} vanish.  Let $A_k=\cup A_{k, i}\subset \hDom$ be such that $A_{k, i}=\{z\in\hDom\mid k=T_{i, \omega}(z)\}$. By construction $A_{i, k}$ is a union of elements of $\hat\xi_i^{\omega}$ and $A_{k, i}\cap A_{k, j}=\varnothing$ for $i\neq j$ (because $T_{i, \omega} <T_{j, \omega}$ for $i<j$ ). For $\Gamma\in \hat\xi_i^{\omega}| A_{k, i}$  by definition of the $\Phi_i$'s we have $\Phi_{k-1}-\Phi_k=\hat \Phi_{i-1}-\hat\Phi_i$. On the other hand, $\Phi_{k-1}\equiv \Phi_k$ on $\hDom\setminus A_k$. Hence for each $k$ and for every $\Gamma\subset A_{k, i}$ we have 
$$
(F_{\sigma^k\omega}^{n-k})_\ast\pi_*(F_\omega\times F_\omega)^{T_{i, \omega}}_*(\Phi_k(m\times m)| \Gamma)=(F_{\sigma^k\omega}^{n-k})_\ast\pi_*'(F_\omega\times F_\omega)^{T_{i, \omega}}_*(\Phi_k(m\times m)| \Gamma)
$$
which finishes the proof of claim. 

It remains to estimate $\int{\Phi_n}d(m\times m)$. We have  
$$
\int{\Phi_n}d(m\times m) = \int_{\{T_{i_1, \omega}>n\}}{\Phi_n}d(m\times m)+\sum_{i=i_1}^{\infty}\int_{\{T_{i, \omega}\le n <T_{i+1, \omega}\}}{\Phi_n}d(m\times m).
$$
Note that $\Phi_n=\Phi$ on $\{T_{i_1, \omega}>n\}$. Hence we have 
$$
\int_{\{T_{i_1, \omega}>n\}}{\Phi_n}d(m\times m)=\int_{\{T_{i_1, \omega}>n\}}{\Phi}d(m\times m)=\bl\{T_{i_1, \omega}>n\}.
$$
Let  $n$ be such that $T_{i, \omega}\le n <T_{i+1, \omega}$. By Lemma \ref{phii} we have $\Phi_n=\hat\Phi_i\le(1-\eps_1)^{i-i_1+1}\Phi$. Hence  
\begin{align*}
\int_{\{T_{i, \omega}\le n <T_{i+1, \omega}\}}{\Phi_n}d(m\times m)\le 
 \int_{\{T_{i, \omega}\le n <T_{i+1, \omega}\}}{(1-\eps_1)^{i-i_1+1}\Phi}d(m\times m)
 \\=(1-\eps_1)^{i-i_1+1}\bl{\{T_{i, \omega}\le n <T_{i+1, \omega}\}}.
\end{align*}
\end{proof}
\section{Decay of correlation} \label{sec:decay2}
The main result of this section is the following  proposition.
\begin{proposition}\label{rates}
For every  $\delta>0$ there is a full measure set $\Omega_6\subset \Omega$ and a random variable  $n_6(\omega)$ \replaced[id=WB]{}{as in Proposition  \ref{prop:tail}} such that for all probability measures $ \lambda_\omega , \lambda_\omega'$  on $\{\Delta_\omega\}$  with $\frac{d\lambda}{dm}, \frac{d\lambda'}{dm}\in \mathcal F_\gamma^{K_\omega}\cap \mF_\gamma^+$, there is $C_{\lambda, \lambda'}$ so that  for any $n > n_6$,  we have 
$$|(F_\omega^n)_*(\lambda)-(F_\omega^n)_*(\lambda')|\ \leq C_{\lambda, \lambda'} \frac{(\log n)^{b\added[id=WB]{+a}}}{n^{a-1-\delta}}.$$ 
\added[id=WB]{Moreover, there exist $C'>0$, such that 
$$P\{ n_6 > n \} \leq C' e^{-u'n^{v'}},$$
for some $u'>0$ and $0<v'<1$.}
\end{proposition}
 Before proving the proposition, we prove the following auxiliary lemma.
\begin{lemma}\label{sublemma}
There exists $C_{\bl}$ such that for all $i\ge 1$ and for any $\Gamma\in \hat\xi_i^\omega$
$$
\bl\{T_{i+1, \omega}-T_{i, \omega}>n |\Gamma\}\le C_{\bl}(m\times m)\{T_{\sigma^{T_{i, \omega}}\omega}>n\}.
$$
\end{lemma}
\begin{proof} 
By definition we have 
$$\bl\{T_{i+1, \omega}-T_{i, \omega}>n |\Gamma\} =\bl\{T_{\sigma^{T_{i, \omega}}\omega}\circ \hF^{i}_\omega> n\}.$$
Therefore, it remains to bound the density $d(\hF_\omega^i)_*\bl/d(m\times m)$.
Let $\Gamma\in \hat\xi_i^\omega$. Any $z', u'\in \hat\Delta_{0, \sigma^{T_{i, \omega}}\omega}$   have unique pre-images $u, z\in \Gamma$. By definition we have 
\begin{align*}
\left|\log \left(\frac{d(\hF^i_\omega)_*\bl}{d(m\times m)}(u') \middle/ \frac{d(\hF^i_\omega)_*\bl}{d(m\times m)}(z')\right)\right|=
 \left| \log\frac{J\hF_\omega^i(z)}{J\hF_\omega^i(u)} +\log \frac{\Phi(u)}{\Phi(z)}\right| \\
\le\hD\gamma^{\hat s(\hF^i_\omega u, \hF^i_\omega z)} + C_\Phi \gamma^{\hat s(u, z)} \le \hD+C_{\Phi}=:\log C_{\bl}.
\end{align*}
Since $\hD$ is independent of $\Gamma$ this implies $d(\hF_\omega^i)_*\bl/d(m\times m)< C_{\bl}$.
\end{proof}
\begin{proof}[Proof of Proposition \ref{rates}] 
Note that, by taking $T_{0, \omega}\equiv 0$    Lemma \ref{lem:main} implies 
\begin{equation}
|(F_\omega^n)_*(\lambda)-(F_\omega^n)_*(\lambda')|\le 2(1-\eps_1)^{1-i_1}\sum_{i=0}^\infty(1-\eps_1)^{i}\bl\{T_{i, \omega}\le n <T_{i+1, \omega}\}.
\end{equation}
By choosing  $A(n)\in \mathbb N$ so that $(1-\eps_1)^{A(n)}\le n^{-2a}$,  for any $i\ge A(n)$ we have 
\begin{equation}\label{eq:igeAn}
\sum_{i=A(n)}^\infty(1-\eps_1)^{i}\bl\{T_{i, \omega}\le n <T_{i+1, \omega}\}\le \sum_{i=A(n)}^\infty(1-\eps_1)^{i}\le \frac{1}{\eps_1n^{2a}}.
\end{equation}
Now we estimate $\bl\{T_{i-1, \omega}\le n <T_{i, \omega}\}$ for  $i\le A(n)$. Let $\bm =\frac{m\times m}{m(\Dom)^2}$.   We proceed as in equation \eqref{eq:Y1expansion1} in the estimate of $Y_1$.    For every $i$ we write 
\begin{equation}\label{eq:ileAn}
\begin{aligned}
&\bl\{T_{i-1, \omega}\le n <T_{i, \omega}\}\le \sum_{\substack{\Gamma\in \hat\xi^\omega_{i-1}\\ T_{i-1, \omega}\mid \Gamma\le n}}\sum_{j=0}^{i-1}\bl\{T_{j+1, \omega}- T_{j, \omega} >\frac{n}{i}|\mid\Gamma\}\bl(\Gamma) \\
&\quad(\text{by Lemma \ref{sublemma}})\quad\le C_{\bl}  \sum_{\substack{\Gamma\in \hat\xi^\omega_{i-1}\\ T_{i-1, \omega}\mid \Gamma\le n}}\bl(\Gamma)\sum_{j=0}^{i-1}(m\times m)\{T_{\sigma^{T_{j, \omega}}\omega} >\frac{n}{i}\}\\
&\le C_{\bl}M^2\sum_{\substack{\Gamma\in \hat\xi^\omega_{i-1}\\ T_{i-1, \omega}\mid \Gamma\le n}}\bl(\Gamma)\sum_{j=0}^{i-1}\bm\{T_{\sigma^{T_{j, \omega}}\omega} >\frac{n}{i}\}. 
\end{aligned}
\end{equation}
\added[id=WB]{Recall that there is a full measure set $\Omega _5$ and a random variable $n_5$ which is finite on $\Omega_5$, and $P\{n_5>n\}\le Ce^{-un^v}$. Now define
$$n_6(\omega) = \inf\{ n | \forall k\geq n, \forall N \in [1, k]\cap \N, n_5( \sigma^{N} \omega ) \leq k\}.$$
We now show that $n_6$ has a stretched exponential tail. Indeed, }
\begin{equation*}\begin{split}
\added[id=WB]{P\{ n_6(\omega) >n\} }&\leq \sum_{k>n} \sum_{N=1}^k P\{ n_5( \sigma^{N} \omega ) >k\}\\ 
&= \sum_{k>n} \sum_{N=1}^k  P\{ n_5(\omega  ) >k\} \le C' e^{-u'n^{v'}},
\end{split}
\end{equation*}
\added[id=WB]{for an appropriate choice of $u'>0$, $v'\in(0,1)$. 
Since $A(n)\sim \log n$, for any $\eps>0$ we have
$$ P\{n_6(\omega)>\frac{n}{A(n)}\}\le P\{n_6(\omega)>n^{1-\eps}\}\lesssim e^{u'n^{(1-\eps)v'}}.$$
Therefore, by Proposition \ref{prop:tail} and the definition of $A(n)$, for any $\frac{n}{A(n)}>n_6$ we can estimate \eqref{eq:ileAn} as follows:}
\begin{equation}\label{eq:goodsetT}
\begin{split}
\bl\{T_{i-1, \omega}\le n <T_{i, \omega}\} 
& \le C_{\bl}M^2C\sum_{j=0}^{i-1} \frac{(\log n )^b}{n^{a-1-\delta}}i^{a-1}=C_{\bl}M^2C\frac{(\log n )^b}{n^{a-1-\delta}}i^{a}\\
&\lesssim (\log n)^{a} \frac{(\log n )^b}{n^{a-1-\delta}} \lesssim \frac{(\log n )^{b+a}}{n^{a-1-\delta}}.
\end{split}
\end{equation}
Finally, using \eqref{eq:goodsetT}
\begin{equation}\label{eq:goodsetTT}
\sum_{i=1}^{A(n)}(1-\eps_1)^{i}\bl\{T_{i, \omega}\le n <T_{i+1, \omega}\}\lesssim \frac{(\log n )^{b+a}}{n^{a-1-\delta}} \sum_{i=1}^{\infty}(1-\eps_1)^{i}\lesssim \frac{(\log n )^{b+a}}{n^{a-1-\delta}}.
\end{equation}
\added[id=WB]{Thus, combining \eqref{eq:igeAn} and \eqref{eq:goodsetTT} finishes the proof.}
\end{proof}

\subsection{Decay of future correlations (Proof of Theorem \ref{thm:DC} item (i))}\label{sec:future_corr}
Let  $\psi\in \mF_\gamma^{K_\omega}$   and $\vp\in \mL_\infty^{K_\omega}$.  Also, let $C_\psi$, $C_\psi'$ and $C_\vp'$ be the constants given in the definitions of $\mF_\omega^{K_\omega}$ and  $\mL_\infty^{K_\omega}$ respectively. Let $\tilde\psi=A_\omega(\psi+(C_\psi' +1) K_\omega+1)$, where $A_\omega = (\int \psi dm +m(\Dom)[(C_\psi'+1)K_\omega+1])^{-1}$. Then  $\tilde\psi  \in \mF^+_\gamma\cap \mF_\gamma^{K_\omega}$,  $\int\tilde\psi dm=1$ and $|\tpsi(x)|\le 2(C_\psi'+1)$ for all $x\in \Dom$.  The second assertion is obvious by the choice of $A_\omega$. For the third one we use the inequality $A_\omega\added[id=WB]{\le m(\Dom)^{-1}(1+K_\omega)^{-1}\le 1}$. For the first claim we have 
$$
\left| \frac{\tilde\psi(x)}{\tilde\psi(y)}-1\right|  \le \frac{1}{K_\omega+1} \left| \psi(x) - \psi(y)\right|\le C_\psi \gamma^{s(x, y)}.
$$
For the correlations we have the following relation
\begin{equation}\label{eq:corr_relation}
C^{(f)}_{n, \omega}(\vp, \psi)= \frac{1}{A_\omega}C^{(f)}_{n, \omega}(\vp, \tpsi)-m(\Dom)[K_\omega(C'_\psi+1) +1]C^{(f)}_{n, \omega}(\vp, \frac{1}{m(\Dom)}).
\end{equation}

Let $\lambda$ be a probability with density  $\frac{d\lambda}{dm}=\tilde\psi$. 
 Then by Proposition \ref{rates}  for every $n>n_6$  we have 
\begin{equation}\label{eq:corr_tpsi}
\begin{aligned}
\left| C_{n, \omega}^{(f)}(\vp, \tpsi)\right|= &\left| \int (\vp_{\sigma^n\omega}\circ F^{n}_\omega)\tpsi_\omega dm -\int \vp_{\sigma^n\omega} d\nu_{\sigma^n\omega}\int \tpsi_\omega dm\right| 
\\&\le \left| \int(\varphi_{\sigma^n\omega}\circ F^{n}_\omega)d\lambda_\omega - \int \varphi_{\sigma^n\omega} 
d\nu_{\sigma^n\omega}\right|
\\& \le  \sup_{x\in \Delta_{\sigma^{n}\omega}}\vp_{\sigma^n\omega}(x)\cdot |(F_\omega^n)_*\lambda_\omega-(F_\omega^n)_*\nu_\omega |
\\& \le C_\vp' K_{\sigma^n\omega}C_{\lambda, \nu}\frac{(\log n)^{b\added[id=WB]{+a}}}{n^{a-1-\delta}}.
\end{aligned} \end{equation}

Similarly, for the probability measure $\lambda'$ with the constant density $m(\Dom)^{-1}$ we have 
\begin{equation}\label{eq:corr_1}
\begin{aligned}
\left|C_{n, \omega}^{(f)}(\vp, \frac{1}{m(\Dom)} )\right|\le C_\vp' K_{\sigma^n\omega}C_{\lambda', \nu}\frac{(\log n)^{b\added[id=WB]{+a}}}{n^{a-1-\delta}}.
\end{aligned} \end{equation}
Define $C_{\lambda, \lambda'}:=\max\{C_{\lambda, \nu}, C_{\lambda', \nu}\}$. Substituting \eqref{eq:corr_tpsi} and \eqref{eq:corr_1} into  \eqref{eq:corr_relation}, and using the inequality $A_\omega^{-1} \le m(\Dom)[1+(2C_\psi'+1))K_\omega]$ we have 
\begin{equation}
\left|C^{(f)}_{n, \omega}(\vp, \psi)\right| \le C_{\lambda, \lambda'} C_\vp' m(\Dom)[2+K_\omega(2+3C_\psi^{\prime})] K_{\sigma^n\omega} \frac{(\log n)^{b\added[id=WB]{+a}}}{n^{a-1-\delta}}.
\end{equation}

\added[id=WB]{Let  $n_7(\omega) =\inf\{ k \ge n_6(\omega)\mid \forall \ell >k,\ K_{\sigma^\ell\omega}\le \ell^{\delta}\}$. Then 
$$
P\{n_7 >n\} \le P\{n_6> n\}+ \sum_{k\ge n} P\{K_{\sigma^k\omega}>k^{\delta}\} \lesssim e^{-u'n^{v'}}+\sum_{k\ge n} e^{-uk^{v'}} \lesssim Ce^{-u'n^{v'}}.
$$
Now, if $n >n_7$ then }
$$
\left|C^{(f)}_{n, \omega}(\vp, \psi)\right|\le C_{\lambda, \lambda'} C_\vp' M[2+K_\omega(2+3C_\psi^{\prime})]\frac{(\log n)^{b\added[id=WB]{+a}}}{n^{a-1-2\delta}}.
$$
If $n\le n_7$ then we let 
$$
C_\omega=n_7(\omega)^{a}K_\omega\sup_{n\le n_7}K_{\sigma^n\omega}.
$$
Hence, for all $n\in \mathbb N$ we have obtained  
\begin{align*}
\left| \int (\varphi_{\sigma^n\omega}\circ F^{n}_\omega)\psi_\omega dm -\int \varphi_{\sigma^n\omega} d\nu_{\sigma^n\omega}\int \psi_\omega dm\right|  \le C_{\psi, \vp}C_\omega \frac{(\log n)^{b\added[id=WB]{+a}}}{n^{a-1-2\delta}}.
 \end{align*} 

It remains to show $P\{C_\omega> n\}$ has the desired decay rate.  We write 
\begin{align*}
P\{C_\omega> k\} \le P\{K_\omega> k^{1/3}\} + P\{\sup_{n\le n_7(\omega)}K_{\sigma^{n}\omega}>k^{1/3}\} 
+P\{n_7(\omega) > k^{1/(3a)}\}.
\end{align*}
\added[id=WB]{Notice that by the definition of $n_7$ we have 
$$
P\{\sup_{n\le n_7(\omega)}K_{\sigma^{n}\omega}>k^{1/3}\}\le P\{n_7(\omega) >k\}+\sum_{n=1}^k P\{K_{\sigma^n\omega}>k^{1/3}\}.
$$
Hence, we have 
$$
P\{C_\omega >k\}\lesssim  (k+1) e^{-uk^{v/3}}+ e^{-u'k^{\delta v}}+  e^{-u'k^{\delta v/(3a)}} \lesssim  e^{-u'k^{v'/(3a)}}.
$$ 
Then the conclusion of the theorem holds with $u'$  and $v':=\frac{v'}{3a}$. }

\subsection{Decay of past correlations} \label{se:past_corr}
To obtain decay of past correlations we need to prove the results of Sections  \ref{sec:decay1} and  \ref{sec:future_corr} with the corresponding shift on $\omega$. Below we use the notation  $\omega=\sigma^{-n}\omega^{\prime}$ for $\omega^{\prime}\in \Omega$. 
\begin{lemma}\label{lem:return_past} 
Let $\lambda_{\omega'}$ and $\lambda_{\omega'}'$ be two probability measures on $\{\Delta_{\omega^{\prime}}\}$ with densities $\vp, \vp'\in\mF_\gamma^+\cap \mL^{K_{\omega^\prime}}_\infty$. 
 Let $\tilde \lambda = \lambda_{\omega'} \times \lambda_{\omega'}^\prime$. For each $\omega^{\prime}\in\Omega$ let  $\omega=\sigma^{-n}\omega^{\prime}$. Then for any  $i\ge 2 $ and $\Gamma \in \xi_i^\omega,$ where   such that $T_\omega |_\Gamma > \tau^\omega_{i-1}$ we 
have 
$$
\tilde \lambda\{ T_{\omega} > \tau^\omega_i |  \Gamma\} \ge 1-C_{\tilde\lambda} V_{\sigma^{\tau^\omega_{i-1}}\omega}^{\tau^\omega_i - \tau^\omega_{i-1}}.
$$
where $0<C_{\bl}<1$. Dependence of $C_{\tilde\lambda}$ on $\bl$ on can be removed if we only consider $i\ge i_0(\bl)$.
\end{lemma}

\begin{lemma}\label{lem:tailtau_too} Let $C_{\tilde\lambda}$ be as in Lemma \ref{lem:return_est1}.
For each $\omega'\in\Omega$, let $\omega=\sigma^{-n}\omega'$. For every   $i$ and $\Gamma \in \xi_i^\omega$ 
$$\tilde \lambda\{ \tau_{i+1}^\omega - \tau_i^\omega > \ell_0 + n | \Gamma\} 
\leq M_0 MC_{\tilde \lambda}^{-1} \cdot m\{ \hR_{\sigma^{\tom_i + \ell_0}\omega} > n \}.$$
\end{lemma}
\begin{proposition}\label{prop:past_tail} Let $\delta>0$ be given. Let $\lambda_{\omega'}$ and $\lambda_{\omega'}'$ be two probability measures on $\{\Delta_{\omega'}\}$ with densities $\vp, \vp'\in\mF_\gamma^+\cap \mF^{K_{\omega^\prime}}_\infty$. 
Then there exists a constant $\hat C_{\bl}$ and a subset $\Omega_5\subset \Omega$ full measure and a random variable  $n_5(\omega')$ which is finite on $\Omega_5$ such that for any 
$n > n_5$  letting $\omega'=\sigma^{-n}\omega$ we have 
$$ \bl \{ T_{\omega} > n \} \leq  \hat C_{\bl}\frac{(\log n)^b}{n^{a-1-\delta}}.$$
Moreover, there exist $u'>0, 0<v<1$ such that for any $n$ 
$$P\{\omega'\mid  n_5(\omega') > n \} \leq C e^{-u'n^v}.$$
\end{proposition}

\begin{proposition}\label{prop:past_rates}
For every  $\delta>0$ there is a full measure set $\Omega_6\subset \Omega$ and a random variable  $n_6(\omega')$, which is finite on $\Omega_6$ such that for all probability measures $ \lambda_{\omega'} , \lambda_{\omega'}'$  on $\{\Delta_{\omega'}\}$  with $\frac{d\lambda}{dm}, \frac{d\lambda'}{dm}\in \mathcal F_\gamma^{K_{\omega'}}\cap \mF_\gamma^{+}$, there is $C_{\lambda, \lambda'}$ so that  for any $n > n_6(\omega')$ letting $\omega'=\sigma^{-n}\omega$ we have 
$$|(F_{\omega}^n)_*(\lambda)-(F_{\omega}^n)_*(\lambda')|\ \leq C_{\lambda, \lambda'} \frac{(\log n)^{b\added[id=WB]{+a}}}{n^{a-1-\delta}}.$$ 
Moreover, there exist $C>0$, $u'>0$ and $0<v'<1$ such that  $$P\{ n_6 > n \} \leq C e^{-u'n^{v'}}.$$
\end{proposition}

Using the above statements and following the same strategy as in the proof of  future correlations we conclude decay of past correlations. 

\section{Appendix}\label{sec:appendix} 
\subsection{Sub-polynomial tail estimates} 
\begin{lemma}\label{appendix1}
Let $a>1$ and $b>0$.  Then
$$
\sum_{k >n}\frac{(\log k)^b}{k^a} \sim  \frac{1}{a-1}\frac{(\log n)^b}{n^{a-1}}.
$$
\end{lemma}

\begin{proof}
The proof is based on integration by parts.  Since $\frac{(\log x)^b}{x^a}$ is monotonically decreasing on $(C, +\infty)$ for $C$ big enough, we have 
$$
\sum_{k >n}\frac{(\log k)^b}{k^a} \le \int_{n}^{\infty}\frac{(\log x)^b}{x^a}.
$$
Let $K=[b]+1$.  Then first making change of variables $y=\log x$ and the integrating by parts $K$ times we obtain 
\begin{equation}\label{integral}
 \int_{n}^{\infty} \frac{(\log x)^b}{x^a}=\frac{1}{a-1}(\log n)^bn^{1-a}+\sum_{i=2}^{K-1}(a-1)^{-i}\prod_{j=0}^{i-1}(b-j)   + I_k(a, b),
\end{equation}
where $I_k(a, b)=(a-1)^{-K}\prod_{j=0}^{K-1}(b-j) \int_{\log n}^\infty y^{b-K}e^{(1-a)y}dy$.
Since $b-K<0$  we conclude that  $I_ k(a)\le (a-1)^{-K-1}\prod_{j=0}^{K-1}(a-j) n^{1-a}$, 
This shows that  the dominant term in \eqref{integral} is $\frac{1}{a-1}(\log n)^bn^{1-a}$.
\end{proof}

\begin{lemma}\label{appendix2}
Suppose $a >0$ and  $a_k \sim \frac{1}{(\log k)^a}$. Then 
$\sum_{k=2}^n a_k \sim \frac{n}{(\log n)^a}.$
\end{lemma}
\begin{proof} A straightforward estimate, using the fact that $\sum_{k=2}^n \frac{1}{(\log k)^a} \rightarrow \infty$ shows that $\sum_{k=2}^n a_k \rightarrow \infty$ and 
$\sum_{k=2}^n a_k \sim \sum_{k=2}^n \frac{1}{(\log k)^a}$.
We work with the latter sum. An elementary estimate shows
$$ \int_2^{n+1} \frac{dx}{(\log x)^a} \leq \sum_{k=2}^n \frac{1}{(\log k)^a} \leq  \frac{1}{(\log 2)^a} 
+ \int_2^n \frac{dx}{(\log x)^a}.$$
Therefore $\sum_{k=2}^n  \frac{1}{(\log k)^a} \sim \int_2^n  \frac{dx}{(\log x)^a}$.  We now 
estimate the integral. 
\begin{equation}\begin{split}
\int_2^n \frac{dx}{(\log x)^a} &= 
\frac{n}{(\log n)^a} - \frac{2}{(\log 2)^a} + a \int_{2}^{ n} \frac{dx}{(\log x)^{a+1}},\\
\end{split}\end{equation}
using integration by parts. 
The first term above is the claimed rate, and the second term is clearly 
$o\left(\frac{n}{(\log n)^a}\right)$.  We will show the same is true for the third, integral term. 
We first upper-bound as follows:
$$ a \frac{(\log n)^a}{n}\int_{2}^{n} \frac{dx}{(\log x)^{a+1}} \leq 
 \frac{a}{n}\int_{2}^{n} \frac{dx}{\log x}.$$
Now simply estimate the right hand side by
\begin{equation*} \begin{split}
 \frac{a}{n}\int_{2}^{n} \frac{dx}{\log x} &=  \frac{a}{n}\int_{2}^{\sqrt n} \frac{dx}{\log x}
+  \frac{a}{n}\int_{\sqrt n}^{ n} \frac{dx}{\log x}\\
&\leq \frac{a}{\log 2}\frac{ \sqrt{n-2}}{n} + \frac{a}{\log \sqrt n }\frac{n -\sqrt n}{n}.\\
\end{split}
\end{equation*}
Since both terms are $o(1)$ in $n$ we are done. 
\end{proof}

\bibliographystyle{amsplain}


\end{document}